\documentclass{amsart}
\usepackage{latexsym}
\usepackage{amsmath}
\usepackage{amssymb}
\usepackage[all]{xy}

\newtheorem{thm}{Theorem}[section]
\newtheorem{prop}[thm]{Proposition}

\newtheorem{lem}[thm]{Lemma}
\theoremstyle{definition}

\theoremstyle{remark}
\newtheorem{rem}[thm]{Remark}

\newcommand{\K}{{\mathbb K}}
\newcommand{\D}{\text{D}}

\newcommand{\mapright}[1]{%
 \smash{\mathop{%
  \hbox to 1cm{\rightarrowfill}}\limits_{#1}}}
\newcommand{\maprightd}[2]{%
 \smash{\mathop{%
  \hbox to 1.2cm{\rightarrowfill}}\limits^{#1}\limits_{#2}}}
\newcommand{\mapleft}[1]{%
 \smash{\mathop{%
  \hbox to 1cm{\leftarrowfill}}\limits_{#1}}}
\newcommand{\mapleftu}[1]{%
 \smash{\mathop{%
  \hbox to 0.8cm{\leftarrowfill}}\limits^{#1}}}
\newcommand{\maprightu}[1]{%
 \smash{\mathop{%
  \hbox to 1cm{\rightarrowfill}}\limits^{#1}}}
\newcommand{\maprightud}[2]{%
 \smash{\mathop{%
  \hbox to 1cm{\rightarrowfill}}\limits^{#1}_{#2}}}
\newcommand{\mapleftud}[2]{%
 \smash{\mathop{%
  \hbox to 1cm{\leftarrowfill}}\limits^{#1}_{#2}}}

\newcounter{eqn}[section]

\def\theeqn{\textnormal{(\thesection.\arabic{eqn})}}

\def\eqnlabel#1{%
  \refstepcounter{eqn}%
  \label{#1}%
  \leqno{\theeqn}}

\begin{document}

\title[Behavior of the EMSS]{Behavior of the Eilenberg-Moore spectral sequence in derived string topology 
}

\footnote[0]{{\it 2010 Mathematics Subject Classification}: 55P35, 55T20 \\
{\it Key words and phrases.} String topology, relative loop homology, Poincar\'e duality space, Eilenberg-Moore spectral sequence. 


Department of Mathematical Sciences, 
Faculty of Science,  
Shinshu University,   
Matsumoto, Nagano 390-8621, Japan   
e-mail:{\tt kuri@math.shinshu-u.ac.jp} 

D\'epartement de Math\'ematiques
Facult\'e des Sciences,
Universit\'e d'Angers,
49045 Angers,
France 
e-mail:{\tt luc.menichi@univ-angers.fr}

Department of Mathematical Sciences, 
Faculty of Science,  
Shinshu University,   
Matsumoto, Nagano 390-8621, Japan   
e-mail:{\tt naito@math.shinshu-u.ac.jp} 
}

\author{Katsuhiko KURIBAYASHI, Luc MENICHI and Takahito Naito}
\date{}
\maketitle


\begin{abstract}
The purpose of this paper is to give applications of the Eilenberg-Moore type spectral sequence converging to 
the relative loop homology algebra of a Gorenstein space, which is introduced in the previous paper due to the authors. 
Moreover, it is proved that the spectral sequence is functorial 
on the category of simply-connected Poincar\'e duality spaces over a space.   
\end{abstract}

\section{Introduction}
This is a sequel to the paper \cite{K-M-N} due to the same three authors.  
In the previous paper, we have developed general theory of derived string topology, namely string topology on Gorenstein spaces due to 
F\'elix and Thomas \cite{F-T}. One of machinery in derived string topology is the Eilenberg-Moore spectral sequence (EMSS) converging 
the loop homology of a Gorenstein space.  
This paper aims at making explicitly computations of relative loop homology algebras 
of Poincar\'e duality spaces by employing the EMSS.  Moreover, we establish functoriality of the EMSS on appropriate categories. 

In what follows, the coefficients of the (co)homology 
and the singular cochain algebra of a space are in a field $\K$ unless otherwise explicitly stated. Moreover, 
it is assumed that a space has the homotopy type 
of a CW-complex whose homology with coefficients in an underlying field is of finite type. 

Let $f: N \to M$ be a map. By definition, the relative loop space $L_fM$, for which we may write $L_NM$, 
is fits into the pullback diagram
$$
\xymatrix@C30pt@R20pt{ 
L_fM \ar[r]^{} \ar[d] & M^I \ar[d]^{(ev_0,  ev_1)} \\
N \ar[r]_(0.4){(f, f)} & M\times M , 
}
$$
where $ev_t$ stands for the evaluation map at $t$. Suppose that $N$ is a simply-connected Poincar\'e duality space.   
Then the so-called loop product on $H_*(L_NM)$ makes the shifted homology ${\mathbb H}_*(L_NM) :=H_{*+\dim N}(L_NM)$ 
into an associative and unital algebra; see \cite[Remark 2.6 and Proposition 2.7]{K-M-N} and Proposition \ref{prop:unital} below. 
We denote by ${\mathbb H}_*(LM)$ the relative loop homology ${\mathbb H}_*(L_MM)$ if $f : M \to M$ is the identity map.  
Observe that ${\mathbb H}_*(LM)$ is nothing but the loop homology due to Chas and Sullivan \cite{C-S} 
when $M$ is a closed oriented manifold; see \cite{F-T}.   We see that the product on ${\mathbb H}_*(LM)$ is an 
extension of the {\it intersection product} on the shifted homology ${\mathbb H}_*(M):=H_{*+d}(M)$ even if $M$ is a 
Poincar\'e duality space; see Proposition \ref{prop:mor_alg} and the argument at the beginning of the Section 3.

The following theorem is a particular version of  \cite[Theorem 2.11]{K-M-N}. 

\begin{thm} 
\label{thm:loop_homology_ss} 
Let $N$ be a simply-connected Poincar\'e duality space of dimension $d$. 
Let $f : N\rightarrow M$ be a continuous map to a simply-connected space $M$.
Then the Eilenberg-Moore spectral sequence is a right-half plane cohomological
spectral sequence $\{{\mathbb E}_r^{*,*}, d_r\}$ converging to the Chas-Sullivan loop homology 
${\mathbb H}_*(L_NM)$  as an algebra with 
$$
{\mathbb E}_2^{*,*} \cong HH^{*, *}(H^*(M); {H}^*(N))
$$
as a bigraded algebra; that is, there exists a
decreasing filtration $\{F^p{\mathbb H}_*(L_N M)\}_{p\geq 0}$ of 
${\mathbb H}_*(L_N M)$ such that 
${\mathbb E}_\infty^{*,*} \cong Gr^{*,*}{\mathbb H}_*(L_N M)$ as a bigraded algebra, where 
$$
Gr^{p,q}{\mathbb H}_*(L_N M) = F^p{\mathbb
  H}_{-(p+q)}(L_NM)/F^{p+1}{\mathbb H}_{-(p+q)}(L_N M).
$$ 
Here $HH^{*,*}(H^*(M), H^*(N))$ denotes the Hochschild cohomology with the cup product.  
\end{thm}

The original version of the theorem above is applicable to Gorenstein spaces whose class contain the classifying spaces of a connected Lie group,  
Noetherian H-spaces, Poincar\'e duality spaces and hence closed oriented manifolds; see \cite{FHT_G, KMnoetherian}. 
In this paper, we introduce an explicit calculation of the relative loop homology of a Poincar\'e duality space over a space.

In general, it is difficult to compute the Chas-Sullivan loop homolgy ${\mathbb H}_*(LM)$ because the shifted homology is not 
functorial with respect to a map between Poincar\'e duality spaces. 
On the other hand, an important feature of the relative version of the loop homology 
is that it gives rise to a functor between appropriate categories. This is explained below. 

Let $\text{\bf Poincar\'e}_M$ be the category of simply-connected based Poincar\'e duality  
spaces over $M$ and based maps; that is, a morphism from $\alpha_1 : N_1 \to M$ to  $\alpha_2 : N_2 \to M$ is a based map 
$f : N_1 \to N_2$ with $\alpha_1= \alpha_2\circ f$. 
Let $\text{\bf Top}_1^N$ be the category of simply-connected spaces under 
$N$.  We denote by $\text{\bf GradedAlg}_A$ and $\text{\bf GradedAlg}^A$ 
the categories of unital graded algebras over an algebra $A$ and of those under $A$, respectively. 
Assume that $N$ is a simply-connected Poincar\'e duality spaces. Then, as mentioned above,  the loop homology 
$\mathbb{H}_*(L_fM):=H_{*+\dim s(f)}(L_fM)$ comes with the loop product, where $s(f) = N$.
In consequence, 
our consideration in \cite{K-M-N}  permits us to deduce the following remarkable theorem. 

\begin{thm}\label{thm:functors}
{\em (1)} The loop homology gives rise to functors 
$$
\mathbb{H}_*(L_?M):=H_{*+\dim s(?)}(L_?M) :  \text{\bf Poincar\'e}_M^{op}
\to \text{\bf GradedAlg}_{H_*(\Omega M)}
$$
and 
$$
\mathbb{H}_*(L_N?):=H_{*+\dim N}(L_N?) :   \text{\bf Top}_1^N \to \text{\bf GradedAlg}^{{\mathbb H}_*(N)}.  
$$
Suppose further that $M$ is a simply-connected Poincar\'e duality space. Then one has a functor 
$$
\mathbb{H}_*(L_?M) : \text{\bf Poincar\'e}_M^{op}
\to \text{\bf GradedAlg}_{H_*(\Omega M)}^{\mathbb{H}_*(LM)}. 
$$
Here $\text{\bf GradedAlg}_{H_*(\Omega M)}^{\mathbb{H}_*(LM)}$ stands for the category of unital graded algebras 
over the algebra $H_*(\Omega M)$ with the Pontrjagin product and under the loop homology $\mathbb{H}_*(LM)$. 

{\em (2)} The multiplicative spectral sequence in Theorem  \ref{thm:loop_homology_ss} converging to the relative loop homology 
is natural with respect to morphisms in $\text{\bf Poincar\'e}_M$ and $\text{\bf Top}_1^N$; that is, for any morphism 
$\rho$ in $\text{\bf Poincar\'e}_M$ or $\text{\bf Top}_1^N$, there
exists a multiplicative morphism of the spectral sequences
such that the map between the associated bigraded algebras, which $\mathbb{H}_*(L_N?)(\rho)$ or  
$\mathbb{H}_*(L_?M)(\rho)$ gives rise to, coincides with the map on the $E_\infty$-terms up to isomorphism. 
\end{thm}

If $N$ is a closed oriented smooth manifold, part (1) follows easily from
~\cite[Theorem 8, see also Corollary 9 and Proposition 10]{G-S}.
Using~\cite[Theorem 4]{F-T}, it is easy to extend \cite[Theorem 8]{G-S}
to Poincar\'e duality space. Therefore (1) can be proved easily.
But in order to prove part (2), we need to interpret (1) in term of differential torsion product described 
in \cite[Theorem 2.3]{K-M-N}; see the proof of Propositions~\ref{prop:alg_maps}. 
and~\ref{prop:alg_maps3}).

For a map $f : N \to M$ between simply-connected Poincar\'e duality spaces, 
Theorem \ref{thm:functors} enables one to obtain algebra maps 
$$
\xymatrix@C50pt@R20pt{ 
{\mathbb H}_*(L_NN) \ar[r]^{{\mathbb H}_*(L_N?)(f)} & 
{\mathbb H}_*(L_NM) & {\mathbb H}_*(L_MM) \ar[l]_{{\mathbb H}_*(L_?M)(f)}. 
}
$$
These maps provide tools to overcome the difficulty arising from the lack of functoriality in the loop homology.
For example, if $f$ is a smooth orientation preserving homotopy equivalence
between manifolds, in~\cite[Proposition 23]{G-S}, Gruher and Salvatore showed that these two algebras maps are isomorphisms and that their composite coincides
with ${\mathbb H}_*(Lf):{\mathbb H}_*(LN)\to {\mathbb H}_*(LM)$.

Furthermore, we are aware that naturality of the EMSS described in Theorem \ref{thm:functors} (2) plays an important role when determining 
the (relative) loop homology of a homogeneous space; see Proposition \ref{prop:cal2} below. 

The layout of this paper is as follows. In Section 2, 
by making use of the spectral sequence described in Theorem \ref{thm:loop_homology_ss}, 
we compute explicitly the Chas-Sullivan loop homology algebra of a Poincar\'e duality space whose cohomology is generated by a single element.  
Section 3 discusses a method for solving extension problems in the $E_\infty$-term of 
our spectral sequence. The naturality of the spectral sequence
described in Theorem \ref{thm:loop_homology_ss} is discussed in Section 4. 
We prove Proposition \ref{prop:cal2} in the end of the section.  
Section 5 is devoted to computations of the relative loop homology of a homogeneous space and 
a Poincar\'e duality space over $BS^1$.

\section{The EMSS calculations of the loop homology}

In this section, by using the spectral sequence in Theorem \ref{thm:loop_homology_ss} 
and the computation of the Hochschild cohomology of a graded commutative algebra, 
we determine explicitly the loop cohomology of a space whose cohomology is generated by a single element. 

We begin by recalling the definition of a Gorenstein space. 
Let $C^*(M)$ be a singular cochain algebra with coefficients in a field $\K$. By definition, a space $M$ is 
a $\K$-{\it Gorenstein space} of dimension $d$ \cite{FHT_G} 
if 
$$
\dim \text{Ext}_{C^*(M)}^*(\K, C^*(M)) = 
\left\{
\begin{array}{l}
0 \ \ \text{if} \ *\neq d,  \\
1  \ \ \text{if} \  *= d. 
\end{array}
\right.
$$ 
The spectral sequence $\{{\mathbb E}_r^{*,*}, d_r\}$ in Theorem \ref{thm:loop_homology_ss} 
is constructed by dualizing the EMSS $\{E_r, d_r\}$ converging to $H^*(LM)$; see \cite[The proof of Theorem 2.11]{K-M-N}. 
Therefore it is immediate that the EMSS $\{E_r^{*,*}, d_r\}$  collapses at the $E_2$-
term if and only if so does the EMSS $\{{\mathbb E}_r^{*,*}, d_r\}$. 
We thus establish the following theorem. 

\begin{thm}\label{thm:calculations}
Let $M$ be a simply-connected $\K$-Gorenstein space of positive dimension 
whose cohomology with coefficients in $\K$ is generated  by a single element of even degree. 
Then as an algebra,
$$
{\mathbb H}_{-*}(LM; \K) \cong  HH^*(H^*(M; \K), H^*(M; \K)). 
$$
\end{thm}

\begin{rem}\label{rem:poly}
Suppose that $M$ is a simply-connected space whose cohomology with coefficients in a field $\K$ 
is a finitely generated polynomial algebra, say $H^*(M) \cong \K[x_1, ..., x_n]$. Let ${\mathbb H}_*(LM)$ denote 
the shifted homology $H_{*-d}(LM)$, where $d = -\sum_{i=1}^n(\deg x_i -1)$. 
Observe that $M$ is a Gorenstein space of dimension $d$ as seen in Remark \ref{rem:Gorenstein} below. We have  
$$
{\mathbb H}_{*}(LM; \K)^\vee \cong HH^*(H^*(M), H^*(M))
$$ 
as a graded vector space. In fact, by using the Eilenberg-Moore spectral sequence converging to 
$H^*(LM)$ with $E_2^{*,*} \cong \text{Tor}_{H^*(M)\otimes H^*(M)}(H^*(M), H^*(M))$, we see that   
$$
({\mathbb H}_{*}(LM))^\vee = (H_{*-d}(LM))^\vee \cong H^{*-d}(LM)) 
\cong (\K[x_1, ..., x_n]\otimes  \wedge(u_1, ..., u_n))^{*-d}
$$
as graded vector spaces, where 
$\deg u_i = \deg x_i-1$. 
Moreover, it follows from  \cite[Theorem 1.1]{Kuri2011} that 
$$
HH^*(H^*(M), H^*(M))\cong HH^*(C^*(M), C^*(M))\cong \K[x_1, ..., x_n]\otimes \wedge(u_1^*, ...., u_n^*)
$$
as algebras, where $\deg u_i ^*= -(\deg x_i-1)$. 
We define a map 
$$\eta : HH^*(H^*(M), H^*(M)) \to H^{*-d}(LM)$$ by 
$$
\eta(x_{i_1}\cdots x_{i_s}u_{j_1}^*\cdots u_{j_t}^*) = x_{i_1}\cdots x_{i_s}u_1 \cdots \widehat{u_{j_1}} 
\cdots \widehat{u_{j_t}}\cdots u_n, 
$$
where $\widehat{u_j}$ means deletion of the element $u_j$ from the representation.  
Then it is readily seen that $\eta$ is an isomorphism of graded vector spaces. See \cite[Section 9]{Menichi} 
for such an isomorphism in more general setting. 
\end{rem}
\begin{rem}
\label{rem:Gorenstein}
Let $M$ be the same space as in Remark \ref{rem:poly}. 
Then $M$ is 
a $\K$-Gorenstein space of dimension $d=-\sum_{i=1}^n(\deg x_i-1)$. 
In fact, since $M$ is a $\K$-formal, it follows that 
\begin{eqnarray*}
\text{Ext}_{C^*(M)}^*(\K, C^*(M)) &\cong& \text{Ext}_{H^*(M)}^*(\K, H^*(M)) \\
&\cong&
(\otimes_{i=1}^{n}\text{Ext}_{\K[x_i]}^*(\K, \K[x_i]))^* = \left\{
\begin{array}{l}
\K \ \ \text{if} \ *\neq d,  \\
0  \ \  \ \text{if} \  *= d. 
\end{array}
\right. 
\end{eqnarray*}
The result \cite[Theorem 6.10]{F-H-T} allows us to obtain the first isomorphism. 
The proof of \cite[(4.6)]{FHT_G} gives us the second one.  

We can choose a shriek map 
$$\Delta^! \in \text{Ext}_{C^*(M^{\times 2})}^d(C^*(M), C^*(M^{\times 2}))=
\text{Ext}_{C^*(M^{\times 2})}^d(C^*(M^I), C^*(M^{\times 2}))=H^0(M)
$$
so that $H(\Delta^!)$ is the integration along the fibre of the fibration 
$\Omega M \to M^I \to M^{\times 2}$. Thus ${\mathbb H}_{*}(LM)^\vee \cong H^{*-d}(LM)$ is endowed with 
$Dlcop$ the dual to the loop coproduct defined in the Section 1. From Remark \ref{rem:poly}, 
one might expect that, as an algebra,  ${\mathbb H}_{*}(LM)^\vee$ is isomorphic to $HH^*(H^*(M), H^*(M))$.
The consideration of such an isomorphism is  one of main topics in \cite{K-L}. 
We also mention that the dual to the loop product on 
${\mathbb H}_{*}(LM)^\vee$ is trivial; 
see \cite{K-L} for more details. 
\end{rem}

As seen in Remark \ref{rem:Gorenstein}, 
a simply-connected space $M$ is a $\K$-Gorenstein space of negative degree 
if the cohomology $H^*(M; \K)$ is a polynomial algebra. 
Then in order to prove Theorem \ref{thm:calculations}, it suffices to consider the case where $H^*(M; \K)$ is a truncated polynomial  algebra and hence 
$M$ is a Poincar\'e duality algebra; see \cite[Theorem 3.1]{FHT_G}
Let $\{{\mathbb E}_r^{*,*}, d_r\}$ be the EMSS converging to ${\mathbb H}_{-*}(LM; \K)$. 
We first observe the following fact. 

\begin{lem} \label{lem:collapsing} Suppose that $H^*(M; \K)$ is a truncated polynomial algebra generated by 
a single element. Then the EMSS 
$\{{\mathbb E}_r^{*,*}, d_r\}$ collapses at the $E_2$-term. 
\end{lem}

\begin{proof}
The proof of \cite[Theorem 2.2]{K-Y} implies that the EMSS $\{E_r, d_r\}$ collapses at the $E_2$-term 
and hence so does $\{{\mathbb E}_r^{*,*}, d_r\}$; see also \cite[Remark 2.6]{K-Y}. 
\end{proof}

We are left to compute the $E_2$-term and to solve all extension problems on ${\mathbb E}_\infty^{*,*}$.  

Let $\K$ be an arbitrary field and $A$ a truncated polynomial algebra of the form 
$\K[x]/ (x^{n+1})$, where $|x|=2m$.

We recall here the calculations of the Hochschild cohomology ring of $A$ due to Yang \cite{Yang}. 
In what follows,  let ${\rm ch} (\K)$ stand for the characteristic of a field $\K$. 

\begin{thm}[{\cite[Theorems 4.6 , 4.7 and 4.8]{Yang}}]\label{thm:Yang_main} {\rm (i)}
If  $n+1\not\equiv 0$ modulo ${\rm ch}(\K)$, then  
$$
HH^{*}(A;A)\cong {\mathbb K}[x,u,t]/(x^{n+1},u^{2},x^{n}t,ux^{n})
$$
as a graded algebra, where $|x|=2m$, $|u|=1$ and $|t|=-2m(n+1)+2$. \\
{\rm (ii)} If ${\rm ch} (\K)\neq 2$ and $n+1\equiv 0$ modulo ${\rm ch} (\K)$, then
$$
HH^{*}(A;A)\cong\K[x,v,t]/(x^{n+1},v^{2})
$$
as a graded algebra, where $|x|=2m$, $|v|=-2m+1$ and $|t|=-2m(n+1)+2$. \\
{\rm (iii)} If ${\rm ch} (\K)=2$ and $n$ is odd, then 
$$
HH^{*}(A;A)\cong\K[x,v,t]/(x^{n+1},v^{2}-\frac{n+1}{2}tx^{n-1})
$$
as a graded algebra, where $|x|=2m$, $|v|=-2m+1$ and $|t|=-2m(n+1)+2$.
Especially, when $n=1$, as a graded algebra,
$$
HH^{*}(A;A)\cong\K[x,v,t]/(x^{2},v^{2}-t)\cong \wedge (x)\otimes \K [v].
$$
\end{thm}

\begin{rem}\label{rem:bidegree}
In view of the $2$-periodic resolution used in the proof of \cite[Main Theorem]{Yang}, we see that 
$\text{bideg}~x=(0, 2m)$,  $\text{bideg}~u=(1, 0)$, $\text{bideg}~v=(1, -2m)$ and 
$\text{bideg}~t=(2, -2m(n+1))$ for the generators $x$, $u$, $v$ and $t$ in $HH^*(A; A)$; 
see \cite[Proposition 3.1]{Yang} and the proofs of \cite[Proposition 3.6]{Yang} and \cite[Theorem 4.7]{Yang}
for more details. 
\end{rem}


Let $M$ be a simply-connected Poincar\'{e} duality space 
whose cohomology with coefficients in $\K$ is isomorphic to $A$ as an algebra. 

\begin{thm}\label{thm0.1}
If $n+1\not\equiv 0$ modulo ${\rm ch}(\K)$, then
$$
{\mathbb H}_{*}(LM;{\mathbb K})\cong {\mathbb K}[x,u,t]/(x^{n+1},u^{2},x^{n}t,ux^{n})
$$
as a graded algebra, where $|x|=-2m$,  $|u|=-1$ 
and $|t|=2m(n+1)-2$.
\end{thm}
\begin{proof} By virtue of Theorem \ref{thm:Yang_main}(i), we have
$$
{\mathbb E}_{2}^{*,*}\cong {\mathbb K}[x,u,t]/(x^{n+1},u^{2},x^{n}t,ux^{n})
$$
as a bigraded algebra, where ${\rm bideg}~x=(0, 2m)$, ${\rm bideg}~u=(1,0)$ 
and ${\rm bideg}~t=(2, -2m(n+1))$; see Remark \ref{rem:bidegree} and Figure (7.1) below. 
Lemma \ref{lem:collapsing} implies that, as bigraded algebras 
$$
{\mathbb E}_{2}^{p,q}\cong {\mathbb E}_{\infty }^{p,q} \cong Gr^{p,q}{\mathbb H}_{*}(LM)\cong F^{p}{\mathbb H}_{-(p+q)}(LM)/F^{p+1}{\mathbb H}_{-(p+q)}(LM).
$$

In order to solve extension problems, 
we verify that the following equalities hold in ${\mathbb H}_{-*}(LM;{\mathbb K})$: 
(1) $x^{n+1}=0$, (2) $u^{2}=0$, (3) $x^{n}u=0$ and (4) $x^{n}t=0$.
Since there exists no non-zero element in ${\mathbb E}_{2}^{p,q}$ for $p\geq 1$ and  $p+q=2m(n+1)$, 
it is readily seen that the equality $(1)$ holds. We next verify that $(2)$ holds. Suppose that $u^{2}=\sum \alpha_{ijk} x^{i}u^{j}t^{k} \neq 0$ for $\alpha_{ijk} \in {\mathbb K}$, $i<n+1$ and $j=0,1$. 
Since the total degrees of $u^{2}$, $x^{i}$ and $t^{k}$ are even, it follows that $j=0$ and 
hence $u^{2}=\sum \alpha_{i0k} x^{i}t^{k}$. We have
\begin{itemize}
\item $2= 2mi+ (-2m(n+1)+2)k$,
\item $2k\geq 3$
\end{itemize}

On the other hand, these deduce that 
\begin{align*}
0&= 2mi -2mk(n+1)+2k -2\\
  &<2m(n+1) -2mk(n+1)+2k -2= 2(m(n+1)-1)(1-k) < 0,  
\end{align*}
which is a contradiction. 
Thus the equality $(2)$ holds. We see that $(3)$ holds as well. In fact, suppose that $x^{n}u=\sum \alpha_{ijk} x^{i}u^{j}t^{k} \neq 0$ 
for $\alpha \in {\mathbb K}$ and $i<n+1$. For the same reason as above, we have $j=1$; 
that is, $x^{n}u=\sum \alpha_{i1k} x^{i}ut^{k}$. This implies that 
\begin{itemize}
\item $2mn+1=2mi+1+(-2m(n+1)+2)k$,
\item $1+2k\geq 2$. 
\end{itemize}
However these deduce that 
\begin{align*}
0&= 2mi+1+(-2m(n+1)+2)k-2mn\\ 
  &<2m(n+1)+1+(-2m(n+1)+2)k-2mn\\ 
  &=2m(1-k)+2k(1-mn)\leq 0.
\end{align*}
We thus obtain the equality $(3)$. In order to verify that the equality $(4)$ holds, we assume that 
$x^{n}t=\sum \alpha_{ijk} x^{i}u^{j}t^{k} \neq 0$ for $\alpha_{ijk} \in {\mathbb K}$ and $i<n+1$. 
It is readily seen that $j=0$ for dimensional reasons. This enables us to deduce that 
\begin{itemize}
\item $2mn-2m(n+1)+2= 2mi+ (-2m(n+1)+2)k$, 
\item $2k\geq 3$. 
\end{itemize}
Since the natural number $k$ is greater than or equal to $2$, it follows that       
\begin{align*}
0&= 2mi-2mk(n+1)+2k -2mn+2m(n+1)-2\\
  &<2m(n+1)-2mk(n+1)+2k -2mn+2m(n+1)-2\\
  &=2(1-k)(mn-1)+2(2-k)m\leq 0,
\end{align*}
which is a contradiction. 
Thus the equality $(4)$ holds. This completes the proof. 
\end{proof}
$$
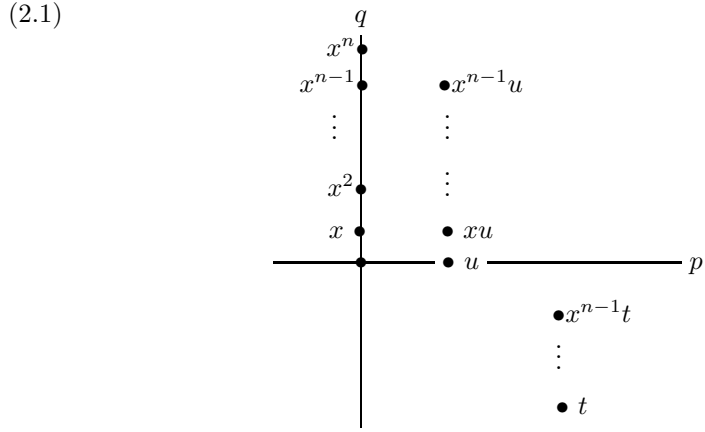
\xymatrix@C5pt@R0.3pt{
             &&    & q \ar@{-}[ddddddddddd]                &           &           &&\\
               &&  & \hspace{-1em} x^{n}  \bullet    &          &               &&\\
               &&  & \hspace{-2em} x^{n-1} \bullet &    \hspace{1.5em}  \bullet  x^{n-1}u    &           &&\\
               &&  & \hspace{-2em} \vdots        &     \hspace{-1em}  \vdots     &           &&\\
                && &\hspace{-1em} x^{2} \bullet  &      \hspace{-1em} \vdots     &        &&\\
                && &\hspace{-1em} x \ \bullet &   \hspace{0.5em}  \bullet \  xu        &            &&\\
\ar@{-}[rrrr] &&& \bullet                         &   \bullet   \ u   \ar@{-}[rrr]    &            &&p\\
               &&  &                                    &       &           && \\
                && &&& \bullet x^{n-1}t&&\\
                && &&&\hspace{-2.5em} \vdots &&\\
                && &&&\hspace{-1.5em} \bullet \ t&&\\
                 &&&&&&& \\
}
\eqnlabel{add-7}
$$
\begin{thm}\label{thm0.2}
If $n+1\equiv 0$ modulo ${\rm ch}(\K)$, $n+1\geq 3$ and ${\rm ch}(\K)\neq 2$, then
$$
{\mathbb H}_{*}(LM; \K)\cong \K[x,v,t]/(x^{n+1},v^{2})
$$
as a graded algebra, where $|x|=-2m$, $|v|=2m-1$ and $|t|=2m(n+1)-2$.
\end{thm}

\begin{proof}
In view of Theorem \ref{thm:Yang_main}(ii), we have  
$
{\mathbb E}_{2}^{*,*}\cong \K[x,v,t]/(x^{n+1},v^{2})
$
as a bigraded algebra, where ${\rm bideg}~x=(0, 2m)$, ${\rm bideg}~v=(1, -2m)$ 
and ${\rm bideg}~t=(2, -2m(n+1))$; see Figure (2.2) below. 
Lemma  \ref{lem:collapsing} yields that, as bigraded algebras  
\[
{\mathbb E}_{2}^{p,q}\cong {\mathbb E}_{\infty }^{p,q} \cong 
Gr^{p,q}{\mathbb H}_{*}(LM)\cong F^{p}{\mathbb H}_{-(p+q)}(LM)/F^{p+1}{\mathbb H}_{-(p+q)}(LM).
\]
We verify that the following equalities hold in ${\mathbb H}_{-*}(LM; \K)$: 
(1) $x^{n+1}=0$ and (2) $v^{2}=0$.
By the same argument as in the proof of Theorem \ref{thm0.3}, it is readily seen that the equality $(1)$ holds. 
Suppose that $v^{2}=\sum \alpha_{ijk} x^{i}v^{j}t^{k} \neq 0$ for $\alpha_{ijk} \in \K$, 
$i<n+1$ and $j=0,1$. Since the total degrees of $v^{2}$, $x^{i}$ and $t^{k}$ are even, we see that $j=0$ and 
hence $v^{2}=\sum \alpha_{i0k} x^{i}t^{k}$. Thus an argument on the total degree and the filtration degree deduces that  
\begin{itemize}
\item $2-4m=2mi+(-2m(n+1)+2)k$,
\item $2k\geq 3$.
\end{itemize}
Then we conclude that 
\begin{align*}
0&= 2mi+(-2m(n+1)+2)k-2+4m\\
  &<2m(n+1)+(-2m(n+1)+2)k-2+4m\\
  &=-2(m(n+1)-1)(k-1)+4m \\
  &\leq -2(3m-1)(k-1)+4m \\
  &\leq  -2(3m-1)+4m =-2m+2\leq 0, 
\end{align*}
which is a contradiction. This completes the proof.
\end{proof}
$$
\xymatrix@C5pt@R0.3pt{
             &&    & q \ar@{-}[ddddddddddd]                &           &           &&\\
               &&  & \hspace{-1em} x^{n}  \bullet    &          &               &&\\
               &&  & \hspace{-2em} x^{n-1} \bullet &    \hspace{1em}  \bullet  x^{n}v    &           &&\\
               &&  & \hspace{-2em} \vdots        &     \hspace{-1em}  \vdots     &           &&\\
                && &\hspace{-1em} x^{2} \bullet  &      \hspace{-1em} \vdots     &        &&\\
                && &\hspace{-1em} x \ \bullet &   \hspace{1em}  \bullet \  x^{2}v        &            &&\\
\ar@{-}[rrrr] &&& \bullet                         &   \hspace{0.5em} \bullet   \ xv   \ar@{-}[rrr]    &            &&p\\
               &&  &                                    &   \bullet \ v    &   \hspace{-1em}  \bullet x^{n}t      && \\
                && &&& \bullet x^{n-1}t&&\\
                && &&&\hspace{-2.5em} \vdots &&\\
                && &&&\hspace{-1.5em} \bullet \ t&&\\
                 &&&&&&&\\
}
\eqnlabel{add-6}
$$

We next consider the case where ${\rm ch}(\K)=2$.

\begin{thm}\label{thm0.3}
If $n$ is odd, ${\rm ch}(\K)=2$ and  $n+1\geq 3$, then
$$
{\mathbb H}_{*}(LM; \K)\cong \K[x,v,t]/(x^{n+1},v^{2}-\frac{n+1}{2}tx^{n-1})
$$
as a graded algebra, where $|x|=-2m$, $|v|=2m-1$ and $|t|=2m(n+1)-2$.
\end{thm}
\begin{proof}
The same  argument as in the proof of Theorem \ref{thm0.2} yields the result.
\end{proof}

By considering the case where $n=1$ and  ${\rm ch}(\K)=2$, namely the cohomology is an exterior algebra, 
we have Theorem \ref{thm:calculations}. 

Suppose that $H^*(M; \K)$ is an exterior algebra generated by a single element. 
For dimensional reasons, we see that the EMSS converging $\{{\mathbb E}_r^{*,*}, d_r\}$ 
to the loop homology ${\mathbb H}_{-*}(LM; \K)$ collapses at the $E_2$-term and that there is no extension problem on the $E_\infty$-term. 
We then establish the following result. 

\begin{thm}\label{thm0.4} Let $M$ be a simply-connected space and $\K$ an arbitrary field. 
Assume that $H^*(M; \K)\cong \wedge (x)$, where $|x|=m$. Then 
$$
{\mathbb H}_{*}(LM; \K) \cong \wedge (x)\otimes \K[v]
$$
as a graded algebra, where $|x|=-m$ and $|v|=m-1$.
\end{thm}
 
\noindent
{\it Proof of Theorem  \ref{thm:calculations}.} Theorems \ref{thm0.1}, \ref{thm0.2}, 
\ref{thm0.3} and \ref{thm0.4} yield the result.  
\hfill\qed 

\begin{rem}
We are aware that Theorems  \ref{thm0.1},  \ref{thm0.2},  \ref{thm0.3} and \ref{thm0.4} recover 
the computations of the loop homology of spheres and complex projective spaces due to 
Cohen, Jones and Yan \cite{C-J-Y} when the coefficients of the homology are in a field. 
\end{rem}

\section{A method for solving extension problems on the EMSS for the loop homology} 

In this section, we give a method for solving extension problems which appear in the first line 
${\mathbb E}_\infty^{0, *}$ of the EMSS converging to the loop homology of a Poincar\'e duality space.

Let $M$ be a simply-connected Poincar\'e duality space of 
dimension $d$ with the fundamental class $\omega_M$. 
We recall that the shifted homology 
${\mathbb H}_*(M) := H_{*+d}(M)$ is an algebra with respect to the intersection pairing $m$ defined by 
$$m(a\otimes b) =(-1)^{d(|a|+d)}(\Delta^!)^\vee(a\otimes b), 
$$
where $\Delta^!$ stands for the shriek map in $\text{Ext}_{C^*(M)}^d(C^*(M), C^*(M\times M))\cong \K$ with 
$(\Delta^!)(\omega_{M})=(\omega_{M\times M})$; see \cite[Theorems 1 and 2]{F-T}.   

Let $[M]$ be the homology class defined by the formula $\langle \omega, [M]\rangle = 1$ with the Kronecker product. 
As seen in the proof of \cite[Theorem 2.11]{K-M-N}, the cap product 
$\theta_{H^*(M)} := \text{-} \cap [M] : H^*(M) \to H_{d-*}(M) = {\mathbb H}_{-*}(M)$ is an isomorphism of algebras; 
see also \cite[Example 10.3 (ii)]{K-M-N}.

\begin{prop}\label{prop:mor_alg} Let $ev_0 : LM \to M$ be the evaluation fibration 
over a simply-connected Poincar\'e 
duality space $M$ and $s : M \to LM$ the section of $ev_0$ defined by $s(x) = c_x$, 
where $c_x$ denotes the constant loop at $x$.  
Then the induced map $s_*  : {\mathbb H}_*(M) \to {\mathbb H}_*(LM)$ is a morphism of algebras. 
\end{prop}

We prove Proposition \ref{prop:mor_alg} after describing our main theorem 
in this section. 

Let $\{{\mathbb E}_r^{*,*}, d_r \}$ be the EMSS converging to the loop homology ${\mathbb H}_*(LM)$, which is 
described in Theorem \ref{thm:loop_homology_ss}. 
The following theorem is reliable when solving extension problems 
on the first line ${\mathbb E}_\infty^{0,*}$. 

\begin{thm} \label{thm:EMSSes}
Let $M$ be a simply-connected Poincar\'e duality space of dimension $d$. Then
{\rm (i)} there exists a first quadrant spectral sequence $\{\widetilde{{\mathbb E}}_r^{*,*}, \widetilde{d}_r \}$ converging to 
the shifted homology ${\mathbb H}_{-*}(M)$ as an algebra such that 
 $\widetilde{{\mathbb E}}_r^{0, *}\cong H^*(M)$ as an algebra and 
 $\widetilde{{\mathbb E}}_r^{i, *}= 0$ for $i > 0$.  \\
{\rm (ii)} There exists a morphism of spectral sequences 
$$
\{s_{r*}\} : \{\widetilde{{\mathbb E}}_r^{*,*}, \widetilde{d}_r \} \to \{{\mathbb E}_r^{*,*}, d_r \}
$$
such that {\rm (a)} each $s_{r*}$ is a morphism of bigraded algebras, 
 {\rm (b)} the diagram 
$$
\xymatrix@C30pt@R10pt{
\widetilde{{\mathbb E}}_2^{0,*} \ar[r]^{s_{2*}} &  {\mathbb E}_2^{0,*}  \\
H^*(M) \ar@{->}[u]^{\cong} &   \\ 
\ar@{=}[u]  H^*({\rm Hom}_{H^*(M)}(H^*(M), H^*(M)) 
\ar[r]_(0.6){H({\rm Hom}(\varepsilon, 1))} & 
HH^{0,*}(H^*(M), H^*(M)) \ar@{->}[uu]_{\cong}
}
$$ 
is commutative, where $H^*(M) \cong \widetilde{{\mathbb E}}_2^{0,*}$ and 
${\mathbb E}_2^{*,*}\cong HH^*(H^*(M), H^*(M))$ are the isomorphisms in {\rm (i)} and in 
Theorem \ref{thm:loop_homology_ss}.  
respectively and \\
{\rm (c)} the map $s_{\infty*}$ coincides with the composite  
$$
\widetilde{{\mathbb E}}_\infty^{0,*} \cong {\mathbb H}_{-*}(M) \stackrel{s_*}{\longrightarrow} 
F^0{\mathbb H}_{-*}(LM)/F^1{\mathbb H}_{-*}(LM)\cong {\mathbb E}_\infty^{0,*}.
$$  
\end{thm}


\begin{rem}\label{rem:cycles} The injective map $ev_0^* : H^*(M) \to H^*(LM)$ factors through 
the edge homomorphism of the EMSS $\{E_r^{*,*}, d_r\}$ converging to the cohomology $H^*(LM)$. 
Observe that the evaluation fibration $p=ev_0: LM \to M$ has a section. 
Thus we see that all the elements in the line $E_2^{0, *}$ survive to the $E_\infty$-term. This implies that 
the elements in ${\mathbb E}_2^{0,*}$ are permanent cycles.  
\end{rem}

\begin{rem}
The relative versions of Proposition \ref{prop:mor_alg} and Theorem \ref{thm:EMSSes} remain valid; that is, the spaces $M$ and $LM$ can be replaced with $N$ and $L_NM$, respectively in the statements. 
This follows from the proofs mentioned below. 
\end{rem}

Before proving  Proposition \ref{prop:mor_alg} and Theorem \ref{thm:EMSSes}, we consider the following  
diagrams 
$$
{\footnotesize
\xymatrix@C2pt@R3pt{
 && M\times M \ar[rrd]_{s\times s} \ar[ddddd]^(0.7){=} \ar[rrrrrrr]^{=}& & & & & & & 
 M \times M \ar[ddddd]^{=} \ar[llld]_{s\times s} \\
 && &  & LM \times LM \ar[rr]^{i}  \ar@{->}'[d][ddd]^(0.3){p\times p}  & & 
 M^I\times M^I \ar@{->}'[d][ddd]^(0.5){p^2} \\
M \ar[rruu]^{\Delta}\ar[rrrd]^{t} \ar[ddddd]^{=} \ar[rrrrrrr]_(0.6){=} & & & & & & & 
M \ar[lld]^{t} \ar[rruu]_{\Delta}\ar[ddddd]^(0.3){=}& \\  
&& & LM\times _MLM \ar[rr]_(0.6){j} \ar[ruu]^(0.6){q} \ar[ddd] & & 
M^I\times_MM^I \ar[ruu]^(0.6){\widetilde{q}} \ar[ddd]^(0.5){u} & \\
 & & & & M\times M \ar@{->}'[r]^(0.6){\Delta \times \Delta}[rr] & & M^{4} \\
 && M \times M \ar[rru]^(0.4){=} \ar@{->}'[r]'[rrr][rrrrrrr]^(0.4){=}& & & & & & & 
 M \times M \ar[lllu]_{\Delta\times\Delta}\\
&& & M \ar[rr]_{(1\times \Delta)\circ \Delta=v} \ar[uur]^(0.6){\Delta} & & 
M^{3}   \ar[ruu]_(0.4){1\times \Delta \times 1=w} & \\
 M \ar[rruu]^{\Delta} \ar[rrru]^{=}\ar[rrrrrrr]_{=}&& & & & & & M, \ar[llu]_v \ar[rruu]_{\Delta} & \\  
}
} 
\eqnlabel{add-0}
$$

$$
{\footnotesize
\xymatrix@C10pt@R3pt{
 &&  M  \ar[lldd]^{=} \ar[rrd]_(0.4){t} \ar[ddddd]^(0.7){=} \ar[rrrrrrr]^{=}& & & & & & &  
 M \ar[ddddd]^{=} \ar[llld]_{t} \ar[lldd]_{=}\\
 && &  & LM \times_M LM \ar[ldd]_(0.4){Comp} \ar[rr]^{j}  \ar@{->}'[d][ddd]^{}  & & 
 M^I\times_M M^I \ar[ldd]_(0.4){Comp} \ar@{->}'[d][ddd]^(0.5){u} \\
M \ar[rrrd]_(0.4){s} \ar[ddddd]^{=} \ar[rrrrrrr]_(0.6){=} & & & & & & & M \ar[lld]^{s} \ar[ddddd]^(0.3){=}& \\  
&& & LM \ar[rr]_(0.6){k}  \ar[ddd] & & M^I \ar[ddd]^(0.5){u} & \\
 & & & & M  \ar@{->}'[r]^-{v}[rr] \ar@{=}[ddl] & & M^{3} \ar[ldd]_(0.4){p_{13}}\\
 && M  \ar[lldd]^{=} \ar[rru]^(0.5){=} \ar@{->}'[r]'[rrr][rrrrrrr]^(0.4){=}& & & & & & & M \ar[lllu]_(0.6){v} \ar[lldd]_{=}\\
&& & M \ar[rr]_{\Delta} & & M\times M    & \\
 M  \ar[rrru]_{=}\ar[rrrrrrr]_{=}&& & & & & & M, \ar[llu]_\Delta  & \\  
}
}
\eqnlabel{add-1}
$$
where $t(x)=(c_x, c_x)$. Observe that all squares in the diagrams are commutative.

\medskip
\noindent
{\it Proof of Proposition \ref{prop:mor_alg}}. 
The commutativity of the left hand-side cube in (3.1) and \cite[Theorem 8.5]{K-M-N} below enable us to deduce that 
$H(\Delta^!)\circ t^* =(s\times s)^*\circ H(q^!)$. By the commutativity of the left-hand side cube in (3.2), we see 
that $t^*\circ Comp^* = s^*$ and hence 
$H(\Delta^!)\circ s^* = (s\times s)^* H(q^!) \circ Comp^*$. This implies that the induced map 
$s_* : {\mathbb H}_*(M) \to {\mathbb H}_*(LM)$ is a morphism of algebras.  
\hfill\qed

\medskip
\noindent
{\it Proof of Theorem \ref{thm:EMSSes}}. The commutative diagrams (3.1) and (3.2) induce a commutative 
diagram 
$$
\hspace{-0.5cm}
{\footnotesize
\xymatrix@C1pt@R10pt{
 &  &  &  \text{Tor}^*_{C^*(M^{2})}(C^*(M), C^*(M^I)) \ar@/_1.5pc/[ldd]_(0.5){\text{Tor}_{\Delta^*}(1, s^*)}  
 \ar@/_2.2pc/[lldd]_(0.6){EM}^(0.6){\cong}  \ar[d]^(0.45){\text{Tor}_{p_{13}^*}(1, c^*)}  \\
 &  & &     \text{Tor}^*_{C^*(M^{3})}(C^*(M), C^*(M^I\times_M M^I)) \ar[ld]^{\text{Tor}_{v^*}(1, t^*)}  \ar@/_2.2pc/
 [lldd]_(0.6){EM}^(0.6){\cong}\\
 & H^*(LM) \ar@/_1pc/[ldd]_(0.5){s^*} \ar[d]_{Comp^*} & \text{Tor}^{*}_{C^*(M)}(C^*(M), C^*(M))  \ar@/^1pc/
 [lldd]_(0.55){EM}^(0.4){\cong} &  \\
  & H^*(LM\times_MLM)  \ar[ld]_(0.6){t^*} \ar@{->}'[d]_{H(q^!)}[dd]  & &   
          \text{Tor}^*_{C^*(M^{4})}(C^*(M), C^*(M^I\times M^I)) \ar@{->}'[l][ll]^(0.25){EM}_(0.25){\cong} \ar[dd]^
          {\text{Tor}_{1}(\Delta^!, 1)} 
            \ar[ld]^(0.35){\text{Tor}_{C^*(\Delta^2)}(1, C^*(s^2))}  
                \ar[uu]_{\text{Tor}_{w^*}(1, {\widetilde{q}}^*)}^{\cong} \\
H^*(M)  \ar[dd]_{H^*(\Delta^!)} & & \text{Tor}^{*}_{C^*(M^{2})}(C^*(M), C^*(M^{2}))  \ar[ll]^{EM}_{\cong} \ar[dd]^
(0.3){\text{Tor}_{1}(\Delta^!, 1)} 
          \ar[uu]_(0.7){\text{Tor}_{\Delta^*}(1, \Delta^*)} & \\
   & H^*(LM)^{\otimes 2} \ar[ld]_{(s^*)^{\otimes 2}}& & \text{Tor}^{*}_{C^*(M^{4})}(C^*(M^{2}), C^*(M^I\times M^I)) 
   \ar@{->}'[l][ll]^(0.4){EM}_(0.4){\cong} 
     \ar[ld]^(0.35){\text{Tor}_{C^*(\Delta^2)}(1, C^*(s^2))}   \\
H^*(M)^{\otimes 2}    & & \text{Tor}^{*}_{C^*(M^{2})}(C^*(M^{2}), C^*(M^{2})).  \ar[ll]^(0.6){EM}_(0.6){\cong} &
}
} 
\eqnlabel{add-2}
$$
The composite of the right hand-side vertical arrows in the big back square is the torsion functor description of 
$Dlp$ in the proof of \cite[Theorem 2.3]{K-M-N}; see the diagram (3.3).   

The Eilenberg-Moore map ${\rm Tor}_{C^*(M)}(C^*(M), C^*(M)) \stackrel{\cong}{\to} H^*(M)$ enables us to  
construct the EMSS converging to $H^*(M)$. 
Dualizing the EMSS, we have a spectral sequence $\{\widetilde{{\mathbb E}}_r^{*,*}, \widetilde{d}_r \}$ 
converging to ${\mathbb H}_*(M)$. 
It is immediate that $\widetilde{{\mathbb E}}_r^{i, *}= 0$ for $i > 0$. 
The front cube (3.2) induces the top square in (3.3), which is commutative, and hence we obtain 
a morphism of spectral sequences $
\{s_{r*}\} : \{\widetilde{{\mathbb E}}_r^{*,*}, \widetilde{d}_r \} \to \{{\mathbb E}_r^{*,*}, d_r \}$. 
Moreover by the commutativity of the diagram (3.3), we see that $\{s_{r*}\}$ satisfies the conditions 
(ii)(a) and (ii)(c). 
In fact, the dual to the composite $\sigma:={\rm Tor}_1(\Delta^!, 1)\circ {\rm Tor}_{\Delta^*}(1, \Delta^*)^{-1}$ 
gives rise to the product on 
each stage $\widetilde{{\mathbb E}}_r^{*,*}$. 

Let $A$ denote the cohomology $H^*(M)$ and $\theta_A =[M]\cap - : A \to A^\vee$ the morphism in 
the proof of \cite[Theorem 2.11]{K-M-N}. 
In order to prove that $s_{2*}$ satisfies (ii)(b), we consider a diagram 
$$
\xymatrix@C50pt@R10pt{
({\rm Tor}_{A^{\otimes 2}}(A, A)^{\vee})^{*-d}  & 
             ({\rm Tor}_{A}(A, A)^{\vee})^{*-d} \ar[l]_{{\rm Tor}_m(1, 1)^\vee} \\
{\rm Hom}_{\K}(H^{d-*}(A\otimes _{A^{\otimes 2}}{\mathbb B}), \K) \ar@{=}[u]  & 
        {\rm Hom}_{\K}(H^{d-*}(A\otimes _{A}A), \K)    \ar[l]^{{\rm Hom}(H(1\otimes\varepsilon),  1)} \ar@{=}[u]  \\
H^{*-d}({\rm Hom}_{\K}(A\otimes _{A^{\otimes 2}}{\mathbb B}, \K)) \ar@{->}[u]^{\cong} & 
  H^{*-d}({\rm Hom}_{\K}(A\otimes _{A}A, \K)) \ar@{->}[u]_{\cong}    
               \ar[l]^{H({\rm Hom}(1\otimes \varepsilon, 1))} \\
 H^{*-d}({\rm Hom}_{A^{\otimes 2}}({\mathbb B},A^{\vee})) \ar[u]^{\iota_{*}}  &
     H^{*-d}({\rm Hom}_{A}(A, A^{\vee})) \ar[u]_{\iota_{*}}  \ar[l]^{H({\rm Hom}(\varepsilon, 1))} \\ 
 H^{*}({\rm Hom}_{A^{\otimes 2}}({\mathbb B},A)) \ar[u]^{\theta_{A*}}  & 
        H^{*}({\rm Hom}_{A}(A, A)) \ar[u]_{\theta_{A*}}   \ar[l]^{H({\rm Hom}(\varepsilon, 1))}\\ 
HH^{*}(A,A) \ar@{=}[u] & A. \ar@{=}[u] \ar[l]^{H({\rm Hom}(\varepsilon, 1))} 
}
\eqnlabel{add-3}
$$
The naturality of maps allows us to deduce that all squares are commutative. 
Observe that the composite of the left hand-side vertical arrows is nothing but the isomorphism 
$\zeta= u \circ  HH^*(1, \text{-}\cap [M] )$ in the proof of \cite[Theorem 2.11]{K-M-N}.  
Thus we obtain the commutative diagram in (ii)(b).  
We are left to prove that $\widetilde{{\mathbb E}}_2^{0, *}\cong H^*(M)$ as an algebra. 
Let $\widetilde{\zeta}$ be the composite of the right hand-side vertical arrows in (3.4). Consider the following diagram 
$$
\hspace{-1cm}
\xymatrix@C25pt@R15pt{
HH^*(A, A)\otimes HH^* (A, A) \ar[d]_{\zeta\otimes \zeta}^{\cong} \ar[r]^(0.6){\cup} & HH^*(A, A) \ar[d]^{\zeta}_{\cong}\\
 \text{Tor}_{A^{\otimes 2}}(A, A)^\vee\otimes  \text{Tor}_{A^{\otimes 2}}(A, A)^\vee \ar[r]^(0.65){(Dlp_2)^\vee}&  
 \text{Tor}_{A^{\otimes 2}}(A, A)^\vee \\
  \text{Tor}_{A}(A, A)^\vee \otimes \text{Tor}_{A}(A, A)^\vee \ar[u]^{\xi\otimes \xi} \ar[r]^(0.6)\mu&  
  \text{Tor}_{A}(A, A)^\vee \ar[u]_{\xi}\\
A\otimes A \ar[u]^{\widetilde{\zeta}\otimes \widetilde{\zeta}}_{\cong} 
\ar[r]^(0.6){\cup}  \ar@/^7pc/[uuu]^(0.5){\eta\otimes \eta} & A 
\ar[u]_{\widetilde{\zeta}}^{\cong}   \ar@/_4pc/[uuu]_(0.5){\eta}\\
}
\eqnlabel{add-6}
$$
in which the center square is commutative, where  $\xi = \text{Tor}_m(1, 1)^\vee$, 
$\eta = \text{Hom}(\varepsilon, 1) : A = H^*(\text{Hom}_A(A, A)) \to HH^*(A,A)$ and $\mu$ 
denotes the dual to the map induced by the composite $\sigma$ mentioned above. 
The map ${\rm Tor}_m(1, 1)$ is an epimorphism and hence $\xi$ is a monomorphism. 
We observe that $\zeta$ is an isomorphism of algebras of degree $-d$; see \cite[Definition 10.1]{K-M-N}.  
Then so is $\widetilde{\zeta}$.  This completes the proof. 
\hfill\qed

\medskip

With the aid of the spectral sequence in Theorem \ref{thm:loop_homology_ss}, 
we show that the relative loop homology of a Poincar\'e duality space is unital. 

\begin{prop} \label{prop:unital} Let $N$ be a simply-connected Poincar\'e duality space of dimension $d$. Then 
the loop homology $\mathbb{H}_*(L_NM)$ is an associative unital algebra. 
\end{prop}

\begin{proof} The result \cite[Proposition 2.7]{K-M-N} yields the associativity of the relative loop homology algebra 
${\mathbb  H}_*(L_NM)$. We prove that the algebra is unital. 

In the rational case, the result follows from \cite[Theorem 2.17]{K-M-N};  
see also \cite[Theorem 1]{F-T:rationalBV}. 
We assume that the characteristic of the underlying field is positive. 

Let $1_N$ stand for the unit of the intersection homology $\mathbb{H}_*(N)$, namely the fundamental class of $N$. 
Let  $s_* : \mathbb{H}_*(N) \to \mathbb{H}_*(L_NM)$ be the algebra map mentioned 
in  Proposition \ref{prop:mor_alg}. 
We put $\mathbb{I} = s_*(1_N)$. 
Then it is immediate that $\mathbb{I}\cdot \mathbb{I} = \mathbb{I}$.  

Recall the right half-plane 
spectral sequence $\{\mathbb{E}_r^{*,*}, d_r\}$ described in 
Theorem \ref{thm:loop_homology_ss}.  
It follows from Remark \ref{rem:cycles} that 
the unit  $1$ in the bigraded algebra $\mathbb{E}_2^{*,*} \cong HH^{*, *}(H^*(M), H^*(N))$ is a permanent cycle. 
Observe that the Hochschild cohomology is unital. 
In view of  Theorem \ref{thm:EMSSes} (ii)(b) and (c), 
we can choose $\mathbb{I}$ as a representative of the unit. In fact 
the diagram (3.4) enables us to deduce that $s_{2*}$ sends the fundamental class to the unit $1$ in  
$\mathbb{E}_2^{*,*}$ up to isomorphism. 

Let $\{F^p\}_{p\geq 0}$ be the filtration of the loop homolog $\mathbb{H}_*(L_NM)$ which the spectral sequence 
$\{\mathbb{E}_r^{*,*}, d_r\}$ provides. 
Then we see that $(F^p)^n=0$ for $p> \dim N -n$; see \cite[Remark 6.1]{K-M-N}.  
This yields that $\mathbb{I}\cdot a = a$ for any $a$ in $(F^p)^n$ with $p=\dim N -n$. Suppose that 
$\mathbb{I}\cdot Q = Q$ for any $Q \in (F^{>s})^n$. Let $\alpha$ be an element in $(F^s)^n$. 
Since $\mathbb{I}\cdot \alpha = \alpha$ in $\mathbb{E}_\infty^{s, *}$, it follows that 
$\mathbb{I}\cdot \alpha = \alpha + R$ for some $R$ in $(F^{s+1})^n$ and hence 
$$
\mathbb{I}\cdot \alpha = (\mathbb{I}\cdot \mathbb{I}) \cdot \alpha 
=\mathbb{I}\cdot (\mathbb{I}\cdot \alpha) = \mathbb{I}\cdot \alpha + R = \alpha + 2R. 
$$
Iterating  the multiplication by $\mathbb{I}$, we see that  $\mathbb{I}\cdot \alpha = \alpha + ch(\K)R = \alpha$.
This completes the proof. 
\end{proof}

\medskip
We now give an application of Theorem \ref{thm:EMSSes}. 

\begin{thm}\label{thm:SO}  Let $M$ be the Stiefel manifold $SO(m+n)/SO(n)$. 
Suppose that $m \leq \text{\em min}\{4, n\}$. Then 
$$
{\mathbb H}_*(LM; {\mathbb Z}/2) \cong \wedge (x_n, x_{n+1}, ..., x_{n+m-1}) \otimes 
{\mathbb Z}/2[\nu_n^*, \nu_{n+1}^*, ..., \nu_{n+m-1}^*]
$$
as an algebra, where $\deg x_i = - i$ and $\deg \nu_j^* = -(1-j)$. 
\end{thm}

We mention that Chataur and Le Borgne \cite{C-L} have determined the loop homology of 
$SO(2+n)/SO(n)$ with coefficients in ${\mathbb Z}$ by using enriched Leray-Serre and Morse 
spectral sequences with the loop product; see \cite[Section 2]{C-L} and \cite[Theorem 2]{LeB}.

\medskip
\noindent
{\it Proof of Theorem \ref{thm:SO}.}
Consider the EMSS $\{{\mathbb E}_r^{*,*}, d_r \}$
converging to ${\mathbb H}_*(LM)$. Since $m \leq n$, it follows that 
$H^*(M; {\mathbb Z}/2) \cong \wedge (x_n, x_{n+1}, ..., x_{n+m-1})$ as an algebra. 
Moreover, the condition that $m\leq 4$ and the proof of \cite[Corollary 5 (1)]{Kuri1991} imply that 
 $\{{\mathbb E}_r^{*,*}, d_r \}$ collapses at the $E_2$-term; see also \cite[Proposition 1.7 (2)]{Kuri1991} and the proof of 
 \cite[Theorem 4]{Kuri1991}. 
 By virtue of \cite[Proposition 2.4]{Kuri2011}, we see that as a bigraded algebra, 
 $$
{\mathbb E}_\infty^{*,*} \cong \wedge (x_n, x_{n+1}, ..., x_{n+m-1}) \otimes 
{\mathbb Z}/2[\nu_n^*, \nu_{n+1}^*, ..., \nu_{n+m-1}^*], 
 $$
 where $\text{bideg}~x_i = (0, i)$ and $\text{bideg}~\nu^*_i = (1, -i)$. 
 
 We solve the extension problems in the $E_\infty$-term.  
 Recall the spectral sequences and the morphism $\{s_{r*}\}$ 
 of spectral sequences in  Theorem \ref{thm:EMSSes}. It follows from 
 Theorem \ref{thm:EMSSes}(ii)(b) and (c) that 
 for the induced map $s_* : {\mathbb H}_*(M) \to  {\mathbb H}_*(LM)$, 
 $s_*(x_i) = x_i$ for any $1 \leq i \leq n+m-1$. 
 Observe that $s_*$ is a morphism of algebras; see Proposition \ref{prop:mor_alg}. 
 It turns out that $x_i^2 =0$ in ${\mathbb H}_*(LM)$ for any $i$. 
 We have the result. 
\hfill\qed

\medskip
\begin{rem}\label{rem:H-space}
Let $X$ be a simply-connected space whose mod $p$ cohomology is an
exterior algebra, 
say $H^*(X; {\mathbb Z}/p)\cong \wedge (y_1, ..., y_l)$. Suppose that either of the following conditions (I) and (II) holds. \\
(I) $X$ is an H-space and $\deg y_i$ is odd for any $i$. \\
(II) $Sq^1 \equiv 0$ if $p=2$. \\ 
Then the same argument as in the proof of Theorem  \ref{thm:SO} enables us to conclude that 
$$
{\mathbb H}_*(LX) \cong
\wedge(\widetilde{y}_1,\widetilde{y}_2, ..., \widetilde{y}_l)\otimes {\mathbb Z}/p[\nu_1^*, \nu_2^*, ...,
\nu_l^*]
$$
as an algebra, where $\deg  \widetilde{y}_j =-\deg y_j$ and 
$\deg \nu_j^* = \deg y_j-1$.

A more general result will appear in~\cite{KMnoetherian}.  
\end{rem}


\section{Naturality of the EMSS}

In order to prove Theorem \ref{thm:functors}, 
we give a correspondence of morphisms between the categories $\text{\bf Poincar\'e}_M^{op}$ and 
$\text{\bf GradedAlg}_{H_*(\Omega M)}$.  
We will describe the proof in terms of the derived tensor products $\text{-} \otimes^{\mathbb L} \text{-} $. 

Let $M$ be a space and $N_{1}$, $N_{2}$ Poincar\'{e} duality spaces of dimension $d_{1}$ and $d_{2}$, respectively. 
For a morphism 
$$
\xymatrix@C25pt@R10pt{
N_{1} \ar[rr]^-{f} \ar[rd]_{\alpha _{1}}& & N_{2} \ar[ld]^{\alpha _{2}} \\
     &  M &
}
$$
in $\text{\bf Poincar\'e}_M$, 
we have a commutative diagram
$$
\xymatrix@C25pt@R20pt{
&L_{\alpha _{2}}M \ar[r] \ar[d]_-{ev_{0}} & LM \ar[d]^-{ev_{0}} \ar[r] & M^{I} \ar[d]^-{(ev_{0},ev_{1})}\\
L_{\alpha _{1}}M \ar[ru]^-{F} \ar@/_0.4pc/[rru] \ar[d]^-{ev_{0}}
&N_{2} \ar[r]^-{\alpha _{2}} & M \ar[r]_-{\Delta } & M\times M \\
N_{1} \ar[ru]^-{f} \ar@/_0.4pc/[rru]_-{\alpha _{1}} &&&
}
$$
for which back squares are pull-back diagrams. 
The singular cochain algebra $C^{*}(N_{i})$ is considered $C^{*}(M^{2})$-module structure via the map $\alpha _{i}^{*}\Delta ^{*}$. 
By \cite[Theorems 1 and 2]{F-T}, we obtain a right $C^{*}(M^{2})$-module map
$
f^{!} : {\mathbb B}_{1} \longrightarrow {\mathbb B}_{2}
$
with degree $d_{2}-d_{1}$.
Here ${\mathbb B}_{i}$ is a right $C^{*}(M^{2})$-semifree resolution of $C^{*}(N_{i})$.
Then, we define a map $F^! :  H^{*}(L_{\alpha _{1}}M) \to H^{*}(L_{\alpha _{2}}M)$ to be the composite
$$
\xymatrix@C25pt@R20pt{
H^{*}(L_{\alpha _{1}}M) \ar[r]^-{{\rm EM}^{-1}}_-{\cong } & H^{*}({\mathbb B}_{1}\otimes _{C^{*}(M^{2})}F) \ar[d]^{H( f^{!} \otimes 1)} & \\
& H^{*}({\mathbb B}_{2}\otimes _{C^{*}(M^{2})}F) \ar[r]^-{{\rm EM}}_-{\cong } & H^{*}(L_{\alpha _{2}}M), &
}
$$
where $\varepsilon :F\to C^{*}(M^{I})$ is a left $C^{*}(M^{2})$-semifree resolution of $C^{*}(M^{I})$.

\begin{prop}\label{prop:alg_maps}
{\em (i)} The shriek map of $F:L_{\alpha_1}M\rightarrow L_{\alpha_2}M$ is compatible with the dual loop product; 
that is, the following diagram is commutative
$$
\xymatrix@C25pt@R15pt{
H^{*}(L_{\alpha _{1}}M) \ar[rr]^-{F^{!}} \ar[d]_-{Dlp} & &H^{*}(L_{\alpha _{2}}M) \ar[d]^-{Dlp}\\
H^{*}(L_{\alpha _{1}}M)\otimes H^{*}(L_{\alpha _{1}}M) \ar[rr]_-{(-1)^{d_{1}(d_{2}-d_{1})}F^{!}\otimes F^{!}} & &H^{*}(L_{\alpha _{2}}M)\otimes H^{*}(L_{\alpha _{2}}M).
}
$$
{\em (ii)} Let $\{E_r^{*,*}, d_r\}$ be the Eilenberg-Moore spectral sequence converging to $H^*(L_{\alpha_i}M)$ in 
Theorem \ref{thm:loop_homology_ss}.  
Then the square
$$
\xymatrix@C25pt@R15pt{
({\rm Tor}_{H^{*}(M^{2})}(H^{*}(N_{2}),H^{*}(M)))^{\vee } \ar[d]_-{(-1)^{d_{1}(d_{2}-d_{1})}{\rm Tor}_{1}(H(f^{!}),1)^{\vee}}& HH^{*}(H^{*}(M),H^{*}(N_{2})) \ar[l]_-{(-1)^{d_{2}}\zeta_2} \ar[d]^-{HH(1,f^{*})} \\
({\rm Tor}_{H^{*}(M^{2})}(H^{*}(N_{1}),H^{*}(M)))^{\vee } & HH^{*}(H^{*}(M),H^{*}(N_{1})) \ar[l]_-{(-1)^{d_{1}}\zeta_1}
}
$$
is commutative, where  $\zeta_i$ is the composite $$u_i\circ HH(1, \theta_{H^*(N_i)}) : 
HH^*(H^*(M), H^*(N_i)) \stackrel{\cong}{\to}  HH^*(H^*(M), H_*(N_i))  \stackrel{\cong}{\to} (E_2^{*,*})^\vee;$$ 
see the proof of \cite[Theorem 2.11]{K-M-N} for the natural map $u_i$.
\end{prop}
\begin{proof} (i) 
By a relative version of \cite[Theorem 2.3]{K-M-N}, we see that the composite 
$$
\xymatrix@C25pt@R10pt{
C^{*}(N_{i})\otimes ^{{\mathbb L}}_{C^{*}(M^{2})}C^{*}(M^{I}) \ar[r]^-{p_{13}^{*}\otimes _{p_{13}^{*}}c^{*}} & C^{*}(N_{i})\otimes^{{\mathbb L}} _{C^{*}(M^{3})}C^{*}(M^{I}\times _{M}M^{I})  \\
& C^{*}(N_{i})\otimes^{{\mathbb L}} _{C^{*}(M^{4})}C^{*}(M^{I}\times M^{I}) \ar[u]_-{\omega^{*}\otimes _{\omega^{*}}\tilde{q}^{*}}^{\simeq } \ar[d]^-{\Delta^{!}\otimes 1}\\
& C^{*}(N_{i}^{2})\otimes ^{{\mathbb L}}_{C^{*}(M^{4})}C^{*}(M^{I}\times M^{I}) \ar[d]^-{{\rm EZ}^{\vee} \otimes _{{\rm EZ}^{\vee}}{\rm EZ}^{\vee}}\\
& (C_{*}(N_{i})^{\otimes 2})^{\vee}\otimes _{(C_{*}(M^{2})^{\otimes 2})^{\vee}}^{\mathbb L} (C_{*}(M^{I})^{\otimes 2})^{\vee}\\
(C^{*}(N_{i})\otimes ^{{\mathbb L}}_{C^{*}(M^{2})}C^{*}(M^{I}) ) ^{\otimes 2} \ar[r]^-{\top }_-{\cong}& C^{*}(N_{i})^{\otimes 2}\otimes ^{{\mathbb L}} _{C^{*}(M^{2})^{\otimes 2}}C^{*}(M^{I})^{\otimes 2} \ar[u]_-{\Theta \otimes _{\Theta }\Theta }^-{\simeq }
}
$$
induces the dual loop product of $H^{*}(L_{\alpha _{i}}M)$. 
Since the morphism $f^{!}$ is in the derived category ${\rm D}({\rm Mod}\text{-}C^{*}(M))$, 
it follows that $f^{!}$ is considered a morphism in ${\rm D}({\rm Mod}\text{-}C^{*}(M^{i}))$ via $p_{13}^{*}$ and $\omega^{*}$ for $i=3,4$. 
This enables us to obtain a homotopy commutative diagram
$$
\xymatrix@C30pt@R10pt{
C^{*}(N_{1})\otimes^{{\mathbb L}} _{C^{*}(M^{2})}C^{*}(M^{I}) \ar[d]_-{p_{13}^{*}\otimes _{p_{13}^{*}}c^{*}} \ar[r]^{f^{!}\otimes 1}& C^{*}(N_{2})\otimes^{{\mathbb L}} _{C^{*}(M^{2})}C^{*}(M^{I}) \ar[d]^-{p_{13}^{*}\otimes _{p_{13}^{*}}c^{*}}\\
 C^{*}(N_{1})\otimes^{{\mathbb L}} _{C^{*}(M^{3})}C^{*}(M^{I}\times _{M}M^{I})  \ar[r]^{f^{!}\otimes 1}&  C^{*}(N_{2})\otimes^{{\mathbb L}} _{C^{*}(M^{3})}C^{*}(M^{I}\times _{M}M^{I}) \\
C^{*}(N_{1})\otimes ^{{\mathbb L}}_{C^{*}(M^{4})}C^{*}(M^{I}\times M^{I}) \ar[u]^-{\omega^{*}\otimes _{\omega^{*}}\tilde{q}^{*}}_{\simeq } \ar[d]_-{\Delta^{!}\otimes 1}  \ar[r]^{f^{!}\otimes 1}& C^{*}(N_{2})\otimes ^{{\mathbb L}}_{C^{*}(M^{4})}C^{*}(M^{I}\times M^{I}) \ar[u]_-{\omega^{*}\otimes _{\omega^{*}}\tilde{q}^{*}}^{\simeq } \ar[d]^-{\Delta^{!}\otimes 1}\\
C^{*}(N_{1}^{2}) \otimes^{{\mathbb L}} _{C^{*}(M^{4})}C^{*}(M^{I}\times M^{I}) \ar[r]^{(f\times f)^{!}\otimes 1}& C^{*}(N_{2}^{2})\otimes ^{{\mathbb L}}_{C^{*}(M^{4})}C^{*}(M^{I}\times M^{I}). 
}
\eqnlabel{add-6}
$$

The fact that $\Delta^{!}f^{!}$ is homotopic to $(f\times f)^{!}\Delta ^{!}$ deduces the commutativity of the bottom square of the diagram (4.1). 
By \cite[Theorem 8.6 (1) and (2)]{K-M-N}, we see that there is a $(C_{*}(M)^{\otimes 2})^{\vee}$-module map $h$ such that the diagrams $$
{\small 
\xymatrix@C30pt@R15pt{
C^{*}(N_{1}^{2})\otimes^{{\mathbb L}} _{C^{*}(M^{4})}C^{*}(M^{I}\times M^{I}) \ar[r]^-{(f\times f)^{!}\otimes 1} \ar[d]_{{\rm EZ}^{\vee}\otimes _{{\rm EZ}^{\vee}}{\rm EZ}^{\vee}} & C^{*}(N_{2}^{2})\otimes^{{\mathbb L}} _{C^{*}(M^{4})}C^{*}(M^{I}\times M^{I}) \ar[d]^-{{\rm EZ}^{\vee}\otimes _{{\rm EZ}^{\vee}}{\rm EZ}^{\vee}} \\
(C_{*}(N_{1})^{\otimes 2})^{\vee}\otimes^{{\mathbb L}} _{(C_{*}(M^{2})^{\otimes 2})^{\vee}} (C_{*}(M^{I})^{\otimes 2})^{\vee} \ar[r]^-{h\otimes 1} & (C_{*}(N_{2})^{\otimes 2})^{\vee}\otimes ^{{\mathbb L}} _{(C_{*}(M^{2})^{\otimes 2})^{\vee}} (C_{*}(M^{I})^{\otimes 2})^{\vee} \\
C^{*}(N_{1})^{\otimes 2}\otimes ^{{\mathbb L}}_{C^{*}(M)^{\otimes 2}}C^{*}(M^{I})^{\otimes 2} \ar[u]^-{\Theta \otimes _{\Theta }\Theta  } \ar[r]_-{(-1)^{d_{1}(d_{2}-d_{1})}(f^{!}\otimes f^{!})\otimes (1\otimes 1)}& C^{*}(N_{2})^{\otimes 2}\otimes ^{{\mathbb L}}_{C^{*}(M)^{\otimes 2}}C^{*}(M^{I})^{\otimes 2}\ar[u]_-{\Theta \otimes _{\Theta }\Theta  }\\
(C^{*}(N_{1})\otimes ^{{\mathbb L}}_{C^{*}(M^{2})}C^{*}(M^{I}) ) ^{\otimes 2} \ar[u]^-{\top } 
\ar[r]_-{(-1)^{d_{1}(d_{2}-d_{1})}(f^{!}\otimes 1)^{\otimes 2}}& (C^{*}(N_{2})\otimes ^{{\mathbb L}}_{C^{*}(M^{2})}C^{*}(M^{I}) ) ^{\otimes 2} 
\ar[u]_-{\top }
}
}
$$
are homotopy commutative. 
Therefore, we have (i) by combining the commutative sqaures mentioned above. \\
(ii) We recall the isomorphism $\theta_{i}^R =[N_i]\cap \text{-} : {H}^*(N_i) \to {H}^*(N_i)^\vee$ of 
{\it right} $H^*(N_i)$-modules with lower degree $\dim N_i$ 
which Poincar\'e duality on $H^*(N_i)$ gives. Since the diagram
$$
\xymatrix@C25pt@R15pt{
H^{*}(N_{1}) \ar[r]^-{H(f^{!})} \ar[d]_-{\theta_{1}^R} & H^{*}(N_{2}) \ar[d]^-{\theta_{2}^R}\\
H^{*}(N_{1})^{\vee } \ar[r]^-{H(f)^{\vee}} & H^{*}(N_{2})^{\vee},
}
$$
is commutative, by \cite[Lemma 10.6]{K-M-N}, we have a commutative diagram 
$$
\xymatrix@C75pt@R15pt{
H^{*}(N_{1})^{\vee} & H^{*}(N_{2})^{\vee} \ar[l]_-{H(f^{!})^{\vee}}\\
H^{*}(N_{1})^{\vee \vee } \ar[u]^-{(\theta_{1}^R)^{\vee}}  & H^{*}(N_{2})^{\vee \vee} \ar[u]_-{(\theta_{2}^R)^{\vee}} \ar[l]_-{(-1)^{d_{2}(d_{2}-d_{1})}H(f)^{\vee \vee}}\\
H^{*}(N_{1}) \ar@/^15mm/[uu]^-{\theta_{H^*(N_1)}} \ar[u]^-{\cong }  & H^{*}(N_{2}), 
\ar@/_15mm/[uu]_-{\theta_{H^*(N_2)}} \ar[l]_-{(-1)^{d_{2}(d_{2}-d_{1})}H(f)} \ar[u]_-{\cong }
}
$$
which is the dual to the above square. This enables us to obtain a commutativity of the diagram 
$$
\xymatrix@C60pt@R15pt{
HH^{*}(H^{*}(M),H^{*}(N_{2})) \ar[r]^{(-1)^{d_{2}(d_{2}-d_{1})}HH(1,H(f))}  \ar[d]_-{{\rm Hom}_{1}(1,\theta_{2})} & HH^{*}(H^{*}(M),H^{*}(N_{1})) \ar[d]^-{{\rm Hom}_{1}(1,\theta_{1})}\\
H({\rm Hom}_{H^{*}(M^{2})}({\mathbb B}, H^{*}(N_{2})^{\vee})) \ar[r]^{H({\rm Hom}_{1}(1,(H(f^{!}))^{\vee}))} \ar[d]_-{\iota_{*}}& H({\rm Hom}_{H^{*}(M^{2})}({\mathbb B}, H^{*}(N_{1})^{\vee})) \ar[d]^-{\iota_{*}} \\
H({\rm Hom}_{{\mathbb K}}(H^{*}(N_{2})\otimes _{H^{*}(M^{2})}{\mathbb B}),{\mathbb K}) \ar[r]^-{H({\rm Hom}_{1}(H(f^{!})\otimes 1,1))} \ar[d]_-{\cong } & H({\rm Hom}_{{\mathbb K}}(H^{*}(N_{1})\otimes _{H^{*}(M^{2})}{\mathbb B}),{\mathbb K}) \ar[d]^-{\cong } \\
(H(H^{*}(N_{2})\otimes _{H^{*}(M^{2})}{\mathbb B}))^{\vee} \ar[r]^{(H(f^{!})\otimes 1)^{\vee}} \ar[d]_-{\cong }& (H(H^{*}(N_{1})\otimes _{H^{*}(M^{2})}{\mathbb B}))^{\vee} \ar[d]^-{\cong }\\
({\rm Tor}_{H^{*}(M^{2})}(H^{*}(N_{2}),H^{*}(M)))^{\vee } \ar[r]^{{\rm Tor}_{1}(H(f^{!}),1)} & ({\rm Tor}_{H^{*}(M^{2})}(H^{*}(N_{1}),H^{*}(M)))^{\vee }. 
}
$$
We have the assertion (ii). 
\end{proof}

We recall the product $m_i$ on the loop homology ${\mathbb H}_{*}(L_{\alpha _{i}}M)=H_{*+d_{i}}(L_{\alpha _{i}}M)$ defined by 
$$
m_{i}(a\otimes b) = (-1)^{d_{i}(|a|+d_{i})}(Dlp)^{\vee}\eta (a\otimes b),
$$
where $\eta :H_{*}(L_{\alpha _{i}}M)^{\otimes 2} \cong (H^{*}(L_{\alpha _{i}}M)^{\vee})^{\otimes 2} \rightarrow (H^{*}(L_{\alpha _{i}}M)^{\otimes 2})^
{\vee}$ is the natural isomorphism.
\begin{prop}\label{prop:alg_maps2}
The map $\widetilde{F^{!}} = (-1)^{d_{1}(d_{2}-d_{1})}(F^{!})^\vee : {\mathbb H}_{*}(L_{\alpha _{2}}M)\to {\mathbb H}_{*}(L_{\alpha _{1}}M)$ is an algebra 
map.
\end{prop}
\begin{proof}
For an element $a\otimes b$ in ${\mathbb H}_{*}(L_{\alpha _{2}}M)^{\otimes 2}$, since $Dlp\circ F^{!}=(-1)^{d_{1}(d_{2}-d_{1})}(F^{!}\otimes F^{!})\circ Dlp$, it follows that  $(F^{!})^{\vee}(Dlp)^{\vee}=(-1)^{d_{1}(d_{2}-d_{1})+d_{2}(d_{2}-d_{1})}(Dlp)^{\vee}(F^{!}\otimes F^{!})^{\vee}$; 
see \cite[Lemma 8.6]{K-M-N} for the sign. Then, we see that 
\begin{align*}
&m_{1}(\widetilde{F^{!}}\otimes \widetilde{F^{!}})(a\otimes b) \\
=& m_{1}((F^{!})^{\vee}(a)\otimes (F^{!})^{\vee}(b))\\
=& (-1)^{d_{1}(|a|+d_{2}-d_{1}+d_{1})} (Dlp)^{\vee}\eta ( (F^{!})^{\vee}(a)\otimes (F^{!})^{\vee}(b) )\\
=& (-1)^{d_{1}(|a|+d_{2})+|a|(d_{2}-d_{1})} (Dlp)^{\vee}(F^{!}\otimes F^{!})^{\vee}\eta (a\otimes b)\\
=& (-1)^{d_{1}(|a|+d_{2})+|a|(d_{2}-d_{1})+d_{1}(d_{2}-d_{1})+d_{2}(d_{2}-d_{1})} (F^{!})^{\vee}(Dlp)^{\vee}\eta (a\otimes b)\\
=& (-1)^{d_{1}d_{2}-d_{1}+d_{2}(|a|+d_{2})}(F^{!})^{\vee}(Dlp)^{\vee}\eta (a\otimes b)\\
=&\widetilde{F^{!}}m_{2}(a\otimes b).
\end{align*}
This completes the proof.
\end{proof}

Let $N$ be a simply-connected Poincar\'{e} duality space of dimension $d$. Consider the commutative diagram of simply-connected spaces
\[
\xymatrix{
M_{1} \ar[rr]^-{g} & & M_{2}\\
 & N. \ar[lu]^-{\beta _{1}} \ar[ru]_-{\beta _{2}}&
}
\] 
Let $\widetilde{g}$ denote $H^*(L_Ng)$, the morphism induced in cohomology
by $L_Ng:L_N M_1\rightarrow L_N M_2$.
Then the map $\widetilde{g}$ coincides with the composite 
\[
\xymatrix@C25pt@R15pt{
H^{*}(L_{\beta _{2}}M_{2}) \ar[r]^-{{\rm EM}^{-1}} & H^{*}({\mathbb B}_{2}'\otimes _{C^{*}(M_{2}^{2})}F_{2}) \ar[r]^-{\overline{id}\otimes _{(g^{2})^{*}}(\overline{g^{I}})^{*}} & H^{*}({\mathbb B}_{1}'\otimes _{C^{*}(M_{1}^{2})}F_{1}) \ar[r]^-{{\rm EM}} & H^{*}(L_{\beta _{1}}M_{1}).
}
\]
Here ${\mathbb B}'_{i}$ is a right $C^{*}(M_{i}^{2})$-semifree resolution of $C^{*}(N)$, $F_{i}$ is a left $C^{*}(M_{i}^{2})$-semifree resolution of 
$C^{*}(M_{i}^{I})$. Moreover,  $\overline{id}$ and $(\overline{g^{I}})^{*}$ are the maps which make the diagrams  
\[
\xymatrix@C25pt@R15pt{
& 
& {\mathbb B}_{1}' \ar[d]^-{\varepsilon }_{\simeq }\\
{\mathbb B}_{2}' \ar[rr]^-{\varepsilon }_-{\simeq } \ar[rru]^-{\overline{id}} & & C^{*}(N),
}
\  \  \  \xymatrix@C25pt@R15pt{
& 
& F_{1} \ar[d]^-{\varepsilon }_-{\simeq }\\
F_{2} \ar[r]^-{\varepsilon }_-{\simeq } \ar[rru]^-{(\overline{g^{I}})^{*}} & C^{*}(M^{I}_{2}) \ar[r]^-{(g^{I})^{*}} & C^{*}(M_{1}^{I})
}
\]
commutative up to homotopy. Observe that $\overline{id}=id$ and $\overline{(g^{I})}^{*}=(g^{I})^{*}$ in the derived categories 
$\D(\text{Mod-}C^{*}(M_{2}^{2}))$ and $\D(C^{*}(M^{2}_{2})\text{-Mod})$, respectively.
\begin{prop}\label{prop:alg_maps3}
{\em (i)} The map $\tilde{g}$ is compatible with the dual loop product.\\
{\em (ii)}  The dual $(\tilde{g})^{\vee}= H_*(L_Ng): {\mathbb H}_{*}(L_{\beta _{1}}M_{1})\to {\mathbb H}_{*}(L_{\beta _{2}}M_{2})$ is an algebra map.
\end{prop}
\begin{proof}
The result \cite[Theorem 8.6 (1) and (2)]{K-M-N} enable us to obtain the commutative squares up to homotopy
$$
{\footnotesize
\xymatrix@C25pt@R15pt{
C^{*}(N)\otimes ^{{\mathbb L}} _{C^{*}(M_{2}^{2})}C^{*}(M^{I}_{2}) \ar[r]^-{1\otimes _{(g^{2})^{*}}(g^{I})^{*}} \ar[d]_-{p_{13}^{*}\otimes _{p_{13}^{*}}c^{*}} & C^{*}(N)\otimes ^{{\mathbb L}} _{C^{*}(M_{1}^{2})}C^{*}(M^{I}_{1}) \ar[d]^-{p_{13}^{*}\otimes _{p_{13}^{*}}c^{*}}\\
C^{*}(N)\otimes ^{{\mathbb L}}_{C^{*}(M_{2}^{3})} C^{*}(M^{I}_{2}\times _{M_{2}}M^{I}_{2}) \ar[r]^-{1\otimes _{(g^{3})^{*}}(g^{I}\times _{g}g^{I})^{*}} & C^{*}(N)\otimes ^{{\mathbb L}}_{C^{*}(M_{1}^{3})} C^{*}(M^{I}_{1}\times _{M_{1}}M^{I}_{1})\\
C^{*}(N)\otimes ^{{\mathbb L}}_{C^{*}(M_{2}^{4})} C^{*}(M^{I}_{2}\times M^{I}_{2}) \ar[r]^-{1\otimes _{(g^{4})^{*}}(g^{I}\times g^{I})^{*}}  \ar[u]^-{\omega^{*}\otimes _{\omega^{*}}\tilde{q}^{*}} \ar[d]_-{\Delta^{!}\otimes 1}& C^{*}(N)\otimes ^{{\mathbb L}}_{C^{*}(M_{1}^{4})} C^{*}(M^{I}_{1}\times M^{I}_{1}) \ar[u]_-{\omega^{*}\otimes _{\omega^{*}}\tilde{q}^{*}} \ar[d]^-{\Delta^{!}\otimes 1}\\
C^{*}(N^{2})\otimes ^{{\mathbb L}}_{C^{*}(M_{2}^{4})} C^{*}(M^{I}_{2}\times M^{I}_{2})  \ar[r]^-{1\otimes _{(g^{4})^{*}}(g^{I}\times g^{I})^{*}} \ar[d]_-{{\rm EZ}^{\vee}\otimes _{{\rm EZ}^{\vee}}{\rm EZ}^{\vee}}& C^{*}(N^{2})\otimes ^{{\mathbb L}}_{C^{*}(M_{1}^{4})} C^{*}(M^{I}_{1}\times M^{I}_{1}) \ar[d]^-{{\rm EZ}^{\vee}\otimes _{{\rm EZ}^{\vee}}{\rm EZ}^{\vee}}\\
(C_{*}(N)^{\otimes 2})^{\vee}\otimes ^{{\mathbb L}}_{(C_{*}(M_{2}^{2})^{\otimes 2})^{\vee}} (C_{*}(M^{I}_{2})^{\otimes 2})^{\vee} \ar[r]_-{h\otimes _{((g^{2})_{*}^{\otimes 2})^{\vee}}((g^{I})_{*}^{\otimes 2})^{\vee}} & (C_{*}(N)^{\otimes 2})^{\vee}\otimes ^{{\mathbb L}}_{(C_{*}(M_{1}^{2})^{\otimes 2})^{\vee}} (C_{*}(M^{I}_{1})^{\otimes 2})^{\vee} \\
C^{*}(N)^{\otimes 2}\otimes ^{{\mathbb L}} _{C^{*}(M_{2}^{2})^{\otimes 2}}C^{*}(M^{I}_{2})^{\otimes 2} \ar[r]^-{1^{\otimes 2}\otimes _{(g^{2})^{*\otimes 2}}(g^{I})^{*\otimes 2}} \ar[u]^-{\Theta \otimes _{\Theta }\Theta }& C^{*}(N)^{\otimes 2}\otimes ^{{\mathbb L}} _{C^{*}(M_{1}^{2})^{\otimes 2}}C^{*}(M^{I}_{1})^{\otimes 2} \ar[u]_-{\Theta \otimes _{\Theta }\Theta }\\
(C^{*}(N)\otimes ^{{\mathbb L}} _{C^{*}(M_{2}^{2})}C^{*}(M^{I}_{2}))^{\otimes 2} \ar[r]^-{(1\otimes _{(g^{2})^{*}}(g^{I})^{*})^{\otimes 2}} \ar[u]^-{\top }& (C^{*}(N)\otimes ^{{\mathbb L}} _{C^{*}(M_{1}^{2})}C^{*}(M^{I}_{1}))^{\otimes 2}. \ar[u]_-{\top }\\
}
}
$$
Thus the commutativity in the torsion functor yields (i). The assertion (ii) is shown by the same argument as in the proof of 
Proposition \ref{prop:alg_maps} (i). 
\end{proof}


\begin{proof}[Proof of Theorem \ref{thm:functors}]
With the same notation as in Propositions \ref{prop:alg_maps2} and \ref{prop:alg_maps3},  
we define a functor ${\mathbb H}_*(L_?M)$ by ${\mathbb H}_*(L_?M)(N)={\mathbb H}_*(L_NM)$ and 
$${\mathbb H}_*(L_?M)(f)=\widetilde{F^!} : {\mathbb H}_*(L_{N_2}M) \to 
{\mathbb H}_*(L_{N_1}M)$$ for a morphism $f : N_1 \to N_2$ in $\text{\bf Poincar\'e}_M$.  
Proposition \ref{prop:unital} implies that ${\mathbb H}_*(L_NM)$ is a unital associative algebra over $H_*(\Omega M)$. 
In fact, the based map $\ast \to M$ gives rise to an algebra map ${\mathbb H}_*(L_NM) \to {\mathbb H}_{*}(L_\ast M)=H_{*+0}(\Omega M) $; 
see Proposition \ref{prop:alg_maps2}. 
It is readily seen that ${\mathbb H}_*(L_?M)(id_N) = id_{{\mathbb H}_*(L_NM)}$. Moreover, the uniqueness of the shriek map enables us 
to deduce that  ${\mathbb H}_*(L_?M)(fg) = {\mathbb H}_*(L_?M)(g)\circ {\mathbb H}_*(L_?M)(f)$; 
see \cite[Theorems 1 and 2]{F-T} and \cite{L-S}. 
Then ${\mathbb H}_*(L_?M)$ is a well-defined functor. 

The result on the naturality of the spectral sequence follows from Proposition \ref{prop:alg_maps} (ii) 
and the proof of Theorem \ref{thm:loop_homology_ss}, namely the construction of the spectral sequence converging to the relative loop homology.  

We define a functor ${\mathbb H}_*(L_N?)$ by ${\mathbb H}_*(L_N?)(M)={\mathbb H}_*(L_NM)$ and 
${\mathbb H}_*(L_N?)(g)=(\widetilde{g})^\vee : {\mathbb H}_*(L_{N}M_1) \to 
{\mathbb H}_*(L_{N}M_2)$ for a morphism $g : M_1 \to M_2$ in $\text{\bf Top}_1^N$.  The well-definedness 
follows from Proposition \ref{prop:alg_maps3}.
\end{proof}

\section{Computational examples in the relative case} 

In this section, we determine explicitly the relative loop homology of a Poinca'e duality spaces by applying  
the EMSS in Theorem \ref{thm:loop_homology_ss} and its functoriality.   

\begin{prop}\label{prop:cal}
Let $f :  M \to K(\mathbb{Z}, 2)=BS^1$ be a map from a simply-connected Poincar\'e duality space $M$.  
Then 
$
\mathbb{H}_*(L_MK(\mathbb{Z}, 2); \K) \cong H^*(M; \K)\otimes \wedge (y)
$
as an algebra, where $\deg x\otimes y = -\deg x +1$ for $x \in H^*(M; \K)$ .
\end{prop}

\begin{proof}
Let $\{ \mathbb{E}_r^{*,*}, d_r \}$ be the 
spectral sequence in Theorem \ref{thm:loop_homology_ss} converging to the algebra 
$\mathbb{H}_*(L_MK(\mathbb{Z}, 2); \K)$. Then it follows from \cite[Proposition 2.4]{Kuri2011} that 
$
\mathbb{E}_2^{*,*} \cong H^*(M; \K)\otimes \wedge (y)
$
as a bigraded algebra, where $\text{bideg} \  y = (1, -2)$. We see that $\mathbb{E}_2^{p,*}=0$ for $* \geq 2$. This yields 
that 
the spectral sequence collapses at the $E_2$-term and that $xy - yx = 0$ 
and $y^2=0$ in $\mathbb{H}_*(L_MK(\mathbb{Z}, 2); \K)$ for any $x\in H^*(M)$. 
Proposition \ref{prop:mor_alg} and Theorem \ref{thm:EMSSes} enable us to solve all extension problems. The answers are trivial.  
We thus have the result. 
\end{proof}

By making use of functors ${\mathbb H}_*(L_?M)$ and ${\mathbb H}_*(L_N?)$, 
we compute the (relative) loop homology of a homogeneous space. 

Let $G$ be a simply-connected Lie group containing $SU(2)$ as a subgroup and $\pi : G \to G/SU(2)$ the projection. 
Suppose that the cohomology $H^*(G; \K)$ is isomorphic to an exterior algebra on generators with odd degree, say $\wedge(V)$. 
Moreover, we introduce the following condition (P): 

\medskip
\noindent
The map $i_* : H_3(SU(2); \K) \to H_3(G; \K)$ induced by the inclusion $i : SU(2) \to G$ is a monomorphism. 

\begin{prop}\label{prop:cal2} With the same assumption on  a Lie group $G$ as above, suppose further that 
the condition {\em (P)} holds.  Then for some decomposition $\K\{y_1\}\oplus \K\{y_2, ...., y_l\}$ of $V$,  
one has a diagram 
$$
\xymatrix@C25pt@R15pt{
{\mathbb H}_*(L(G/SU(2))) \ar[d]_{\rho:={\mathbb H}_*(L_?(G/SU(2))))(\pi)}&  \ar[l]^(0.55){\cong} \wedge(y_2', ...., y_l')\otimes \K[\nu_2^*, ..., \nu_l^*] \\
{\mathbb H}_*(L_G(G/SU(2))) & \ar[l]^(0.55){\cong} \wedge(y_1', y_2', ...., y_l')\otimes \K[\nu_2^*, ..., \nu_l^*] \\
{\mathbb H}_*(LG) \ar[u]^{\rho':={\mathbb H}_*(L_G?)(\pi))}  & \ar[l]^(0.65){\cong} \wedge(y_1', y_2', ..., y_l')\otimes \K[\nu_1^*, \nu_2^*..., \nu_l^*] 
}
$$
in which the horizontal arrows are isomorphisms of algebras, 
$\rho(y_i') = y_i'$, $\rho(\nu_i^*) = \nu_i^*$, $\rho'(y_i') = y_i'$, $\rho'(\nu_1^*) = 0$ and $\rho'(\nu_i^*) = \nu_i^*$ for $i > 1$ up to isomorphism, where 
$\deg \nu_i = \deg y_i - 1$ and $\deg y_i' = -\deg y_i$.  
\end{prop}

\begin{rem}
In general, for a simply-connected compact Lie group $G$, 
the homology $H_3(G; {\mathbb Z})$ with coefficients in ${\mathbb Z}$ is torsion free and its rank coincides with 
the number of simple factors of $G$; see \cite[Theorem 6.4.17]{Mimura-Toda} for example. 

The result \cite[Theorem 6.6.23]{Mimura-Toda} asserts that for any compact, simply-connected simple Lie group $G$ there exists an inclusion 
$SU(2) \to G$ such that the induced map $i_* : H_3(SU(2); {\mathbb Z}) \to H_3(G; {\mathbb Z} )$ is an isomorphism and then 
the condition (P) holds. 
On the other hand, as seen in \cite[p. 767]{Mccleary-Ziller}, there exist  Lie groups containing $SU(2)$ as a subgroup such that the induced map $i_* : H_3(SU(2); {\mathbb Z}) \to H_3(G; {\mathbb Z} )$ is multiplication 
by an integer greater than one.  
Thus we see that the condition (P) does not necessarily hold.   
\end{rem}

\begin{proof}[Proof of Proposition \ref{prop:cal2}]
Let $\{E_r, d_r\}$ and $\{E_r', d_r'\}$ be the EMSS's associated with the fibrations 
$G/SU(2) \to BSU(2) \to BG$ and $G \to EG \to BG$, respectively. We have a morphism of fibrations 
$$
\xymatrix@C25pt@R15pt{
G \ar[r] \ar[d]_{\pi} & EG \ar[r] \ar[d] &  BG \ar@{=}[d] \\
G/SU(2) \ar[r] & BSU(2) \ar[r]_{Bi} & BG . 
}
$$
This induces a morphism $\{f_r\}$ of spectral sequences from $\{E_r, d_r\}$ to $\{E_r', d_r'\}$.
Since the condition (P) holds, it follows that $(Bi)^* : H^4(BG; \K) \to H^4(BSU(2); \K)$ is an epimorphism.  
Therefore there exists a decomposition $\K\{y_1\}\oplus \K\{y_2, ...., y_l\}$ of $V$ such that, as bigraded algebras 
$$E_2^{*,*}\cong \text{Tor}_{H^*(BG)}(\K, H^*(BSU(2)))\cong \wedge (y_2, ...., y_l), $$ 
$${E_2'}^{*,*}\cong \text{Tor}_{H^*(BG)}(\K, \K)\cong \wedge (y_1, y_2, ...., y_l) $$
and $ f_2(y_i) = y_i$.  
The algebra generators of the $E_2$-term in both the spectral sequences are in the second line.  This implies that 
$$H^*(G/SU(2))\cong \wedge (y_2, ...., y_l), $$
$$H^*(G)\cong \wedge (y_1, y_2, ...., y_l)$$
and that $\pi^*(y_i) = y_i$. 
Let $\{{\mathbb E}_r, d_r\}$, $\{{\mathbb E}_r', d_r'\}$ and $\{{\mathbb E}_r'', d_r''\}$ be the spectral sequences converging to 
the loop homology ${\mathbb H}_*(L(G/SU(2)))$, ${\mathbb H}_*(LG)$ and ${\mathbb H}_*(L_G(G/SU(2)))$ in 
Theorem \ref{thm:loop_homology_ss}, respectively. Theorem \ref{thm:functors} (2) and the proof of \cite[Proposition 2.4]{Kuri2011} yield the commutative diagram 
$$
\xymatrix@C15pt@R15pt{
{\mathbb E}_2 \ar[d]_{g_2} &  \ar[l]^(0.85){\cong}  HH^*(H^*(G/SU(2)), H^*(G/SU(2)))\ar@{=}[r]  \ar[d]^{HH(1, \pi^*)} 
     &\wedge(y_2', ...., y_l')\otimes \K[\nu_2^*, ..., \nu_l^*]   \\
{\mathbb E}_2'' & \ar[l]^(0.8){\cong} HH^*(H^*(G/SU(2)), H^*(G))\ar@{=}[r]&\wedge(y_1', y_2', ...., y_l')\otimes \K[\nu_2^*, ..., \nu_l^*] 
   \\
{\mathbb E}_2' \ar[u]^{g_2'}  & \ar[l]^(0.8){\cong}  HH^*(H^*(G), H^*(G))\ar@{=}[r] \ar[u]_{HH(\pi^*, 1)} 
 &\wedge(y_1', y_2', ..., y_l')\otimes \K[\nu_1^*, \nu_2^*..., \nu_l^*] 
}
$$   
for which $g_2(y_i') = y_i'$, $g_2(\nu_i^*) = \nu_i^*$, $g_2'(y_i') = y_i'$, $g_2'(\nu_1^*) = 0$ and $g_2'(\nu_i^*) = \nu_i^*
$ for $i > 1$, 
where $\text{bideg} \ \nu_i = (1, -\deg y_i )$ and $\text{bideg} \ y_i' = (0, \deg y_i)$.  
It follows from Remark \ref{rem:H-space}  that $\{{\mathbb E}_r', d_r'\}$ collapses at the $E_2$-term 
and that there is no extension problem on the $E_\infty$-term. 
This implies that $\{{\mathbb E}_r'', d_r''\}$ collapses at the $E_2$-term and that $y_i^2=0$ in 
${\mathbb H}_*(L_G(G/SU(2)))$.   Moreover, we see that there is no extension problem for 
commutativity and relations between generators since $g_2'$ is an epimorphism.

Since the map $g_2$ is a monomorphism, it follows that  $\{{\mathbb E}_r, d_r\}$ collapses at the $E_2$-term.  
Moreover, the same argument as in the proof of Theorem  \ref{thm:SO} with Theorem \ref{thm:EMSSes} 
yields that there is no extension problem on the $E_\infty$-term of $\{{\mathbb E}_r, d_r\}$. 
This completes the proof.  
\end{proof}


We conclude this section with a comment on the relative loop homology.  
\begin{rem}
The result \cite[Proposition 6.1]{Naito} due to the third author asserts that the relative loop product is not graded commutative in general. 
On the other hand, the explicit calculation shows that 
the loop product on $\mathbb{H}_*(L_MK(\mathbb{Z}, 2); \K)$ is graded commutative for any 
map $f :  M \to K(\mathbb{Z}, 2)=BS^1$ from a simply-connected Poincar\'e duality space; 
see Proposition \ref{prop:cal}. 

For instance, consider the inclusion 
$\mathbb{C}P^n \to K(\mathbb{Z}, 2)$. The result above says that the algebra structure of the loop homology  
$\mathbb{H}_*(L_{\mathbb{C}P^n}K(\mathbb{Z}, 2))$ is comparatively simple than that of 
$L{\mathbb{C}P^n}$; see Theorems \ref{thm0.1}, \ref{thm0.2} and \ref{thm0.3}. 

Since $K(\mathbb{Z}, 2)=BS^1$ is a Gorenstein space of dimension $-1$, the shifted homology 
$\mathbb{H}_*(LK(\mathbb{Z}, 2))=H_{*-1}(LK(\mathbb{Z}, 2))$ 
is endowed with the loop product as mentioned above. However, results \cite[Theorem 4.5 (i)]{Tamanoi:stabletrivial} 
and \cite[Theorem 2.1]{K-L} assert that the loop product on $K(\mathbb{Z}, 2)$ is trivial. It seems that 
the homology invariant, the loop product captures notably variations of the spaces in which whole loops or 
their stating points move. 
\end{rem}


\end{document}